\documentclass[3p,onecolumn]{article}
\usepackage[utf8]{inputenc}
\usepackage[T1]{fontenc}
\usepackage{lmodern}
\usepackage{amsmath}
\usepackage{amsfonts}
\usepackage{amssymb}
\usepackage{amsthm}
\usepackage[english]{babel}
\usepackage{algorithm}
\usepackage{enumitem}
\usepackage{tikz}
\usetikzlibrary{calc}
\usetikzlibrary{arrows}
\usetikzlibrary{decorations.markings}
\usetikzlibrary{intersections}
\usepackage{MnSymbol}

\newcommand{\N}{\mathbb{N}}
\newcommand{\R}{\mathbb{R}}
\newcommand{\Z}{\mathbb{Z}}
\newcommand{\ol}{\overline}
\newcommand{\ul}{\underline}
\newcommand{\G}{G_{n,k,b}}
\newcommand{\V}{V_{n,k,b}}

\newcommand{\Om}{\Omega}

\newtheorem{theorem}{Theorem}
\newtheorem{corollary}{Corollary}
\newtheorem{lem}{Lemma}
\newtheorem{prop}{Proposition}

\newtheorem{prob}{Problem}
\DeclareMathOperator*{\argmin}{arg\,min}
\DeclareMathOperator*{\argmax}{arg\,max}
\DeclareMathOperator*{\inte}{int}
\newcommand{\diam}{\operatorname{diam}}

\begin{document}

%\begin{frontmatter}

\title{Bandwidth of graphs resulting from the edge clique covering problem}

\author{Konrad~Engel\thanks{corresponding author, email address:\
konrad.engel@uni-rostock.de},
Sebastian~Hanisch\thanks{email address:\
sebastian.hanisch@uni-rostock.de}\\
{\small University of Rostock, Institute for Mathematics,
    Ulmenstrasse 69, 18057 Rostock, Germany}}
\date{}
 \maketitle
%\newpage

\begin{abstract}
Let $n,k,b$ be integers with $1 \le k-1 \le b \le n$ and let $G_{n,k,b}$ be the graph whose vertices are the $k$-element subsets $X$ of $\{0,\dots,n\}$ with $\max(X)-\min(X) \le b$ and where two such vertices $X,Y$ are joined by an edge if $\max(X \cup Y) - \min(X \cup Y) \le b$.
These graphs are generated by applying a transformation to maximal $k$-uniform hypergraphs of bandwidth $b$ that is used to
reduce the (weak) edge clique covering problem to a vertex clique covering problem.
The bandwidth of $G_{n,k,b}$ is thus the largest possible bandwidth of any transformed $k$-uniform hypergraph of bandwidth $b$.
For $b\geq \frac{n+k-1}{2}$, the exact bandwidth of these graphs is determined. For $b<\frac{n+k-1}{2}$, the bandwidth is asymptotically determined in the case of $b=o(n)$ and in the case of $b$ growing linearly in $n$ with a factor $\beta \in (0,0.5]$, where for one case only bounds could be found. It is conjectured that the upper bound of this open case is the right asymptotic value.
\end{abstract}

%\begin{keyword}
% bandwidth of graphs  \sep bandwidth numbering \sep edge clique coverings \sep vertex clique coverings \sep $k$-element subsets
%\end{keyword}

%\end{frontmatter}

\section{Introduction}\label{sec1}
The bandwidth problem for graphs is to find a labelling of the vertices with different integers, such that the maximum absolut value of the difference of the labels of two adjacent vertices is minimal.
There are many applications such as efficient storage of sparsely populated symmetric matrices, which arise e.g. from discretization of partial differential equations,
cf. \cite{Pissanetzky1984}.
Several other applications, including the placement problem for modules of a VLSI design, the binary constraint satisfaction problem and the minimization of effects of noise in the multichannel communication of data are discussed e.g. in \cite{CCDG, BerR02, Chu88}. The bandwidth problem was shown to be NP-hard \cite{BW_NPc} and even an approximation with a ratio better than $2$ is NP-hard \cite{DuFeUn2010}, so several heuristics such as the Cuthill-McKee-Algorithm \cite{cuthill} or some similar approaches, cf. \cite{GPS-Algorithm}, are very popular in applications. However, for some graph classes the exact bandwidth is known. These include the path, the cycle, the complete graph, the complete bipartite graph \cite{bwdiamb}, the hypercube \cite{hypercube}, the grid graph \cite{labeltwopath}, special Hamming graphs \cite{HenS92} and several others, cf. \cite{CCDG}. However, there are still many graphs, where the exact bandwidth is unknown, such as the general Hamming graphs, cf. \cite{Har03, BBHS08}. In this paper, we consider graphs $\G$, $1 \le k-1 \le b \le n$, whose vertices are those $k$-element subsets of $\{0,\dots,n\}$, for which the difference of the maximum and the minimum is at most $b$. There is an edge between two vertices, if the difference of the maximum and the minimum of the union of the corresponding sets is at most $b$. We start by introducing the necessary notation and a motivation in Section \ref{sec2} and study some basic properties in Section \ref{170}. Based on that, we determine the exact bandwidth for these graphs in the case of $b\geq \frac{n+k-1}{2}$ in Section \ref{sec3}. In Section \ref{sec4}, we present some asymptotic results for $n\to \infty$ in the case of $b=o(n)$. The results of Sections \ref{sec3} and \ref{sec4} are summarized by the following theorem:
\begin{theorem}\label{thm1}
Let $k$ be a fixed positive integer and $1 \le k-1 \le b \le n$.
\begin{enumerate}[label=\alph{enumi})]
\item If $b\geq\frac{n+k-1}{2}$ then
\[
B(G_{n,k,b})=  \left\lceil\frac{(n+1)\binom{b}{k-1}-(k-1)\binom{b+1}{k}+\binom{2b-n+1}{k}-2}{2}\right\rceil.
\]
\item If $b=o(n^{\frac{1}{k+1}})$ then for sufficiently large $n$
\[
B(G_{n,k,b})=k\binom{b}{k}.
\]
If $b=o(n)$ then
\[
B(G_{n,k,b})\sim k\binom{b}{k} \text{ as } n\rightarrow \infty.
\]
\end{enumerate}
\end{theorem}
Sections \ref{451} to \ref{1122} discuss the case $b\sim\beta n$ with $\beta \in (0,0.5]$. The main result is given by the next theorem:
\begin{theorem}\label{thm2}
Let $k\ge 2$ be a fixed positive integer, but $n\to \infty$.  Let $b\sim \beta n$ and let $1=q\beta +r$, where $q\ge 2$ is a positive integer and $0\leq r<\beta$.
Let
\begin{align*}
c_1(\beta,k)&=\frac{\beta^k}{k!} \left(k-\frac{k-1}{q}\right),\\
c_2(\beta,k)&=\frac{\beta^{k-1}}{(q+1)k!} \left(k-(k-1)\beta\right),\\
c_3(\beta,k)&=\frac{(\beta-r)^k}{(q+1)k!} q^{k-1}.
\end{align*}

\begin{enumerate}[label=\alph{enumi})]
\item If $r\leq \frac{q-1}{q^2+q-1}$ then
$B(G_{n,k,b})\sim c_1(\beta,k) n^k.$
\item If $r> \frac{q-1}{q^2+q-1}$ then
$\max\{c_1(\beta,k), c_2(\beta,k)+\frac{1}{q^{k-1}}c_3(\beta,k)\}n^k \lesssim B(\G)\lesssim
(c_2(\beta,k)+c_3(\beta,k))n^k.$
\end{enumerate}
\end{theorem}

The part b) gives only bounds instead of an exact asymptotic value. We strongly conjecture that the RHS bound is the right value.
The bounds are not too far away from each other because
\[
\frac{c_2(\beta,k)}{c_3(\beta,k)} = \left(\frac{\beta}{q\beta-qr}\right)^k \frac{q}{\beta} (k-(k-1) \beta) \ge \left(\frac{\beta}{q\beta-qr}\right)^k \frac{k(q-1)+1}{\beta} \ge 6
\]
since $\beta > q\beta-qr$ iff $r >\frac{q-1}{q^2+q-1}$ and $ k \ge 2, q \ge 2, \beta \le \frac{1}{2}$.

Let $U=\{\beta \in (0,0.5]: r > \frac{q-1}{q^2+q-1}\}$ be the set of numbers $\beta$ for which part b) applies and thus the exact asymptotic value is still unknown. Note that $r > \frac{q-1}{q^2+q-1}$ iff $\frac{1}{q+1} < \beta < \frac{q}{q^2+q-1}$. Thus the Lebesgue measure of $U$ is equal to $\sum_{q=2}^{\infty} (\frac{q}{q^2+q-1}-\frac{1}{q+1}) = 0.119\dots$, i.e., for the ``majority'' of numbers $\beta \in (0,0.5]$ the exact value is known.

The proof of Theorem \ref{thm2} is based on a reduction to a continuous problem on the unit square $[0,1]^2$. Riemann integrals and elementary geometric arguments suffice. The embedding into a more difficult continuous problem on the unit cube was used by Harper \cite{Har99} to obtain bounds for the bandwidth of Hamming graphs. Also for the edge-bandwidth of multidimensional grids and Hamming graphs (the bandwidth of the line graph of these graphs) Harper's reduction to the unit cube was applied in \cite{AJM08}. Asymptotic bounds for the bandwidth of the $d$-ary de Bruijn graph were obtained in \cite{Pel93} by an approach based on the use of a continuous domain.

\section{Notation and motivation}\label{sec2}
Let $[n]=\{1,\dots, n\}$ and $[i,j]=\{i,i+1,\dots,j-1,j\}$ with $i,j\in \Z, \, i\leq j$. In particular, $[0,n]=\{0,1,\dots,n\}$.
For a graph $G=(V,E)$ with $|V|=n$ vertices, a \emph{proper numbering} of $G$ is a bijection $f:V\to [n]$. For two vertices $u,v\in V$, we call $d_f(u,v)=|f(u)-f(v)|$ the \emph{$f$-distance} of $u$ and $v$. Let $f$ be a proper numbering of a graph $G$. The \emph{bandwidth of} $f$, denoted $B_f(G)$, is given by $B_f(G)=\max\{|f(u)-f(v)|:\, \{u,v\}\in E\}$, i.e., the maximal $f$-distance. The \emph{bandwidth} of $G$ is defined by $B(G)=\min\{B_f(G): \, f \text{ is a proper numbering of } G\}$. A \emph{bandwidth numbering} of $G$ is a numbering $f$ such that $B(G)=B_f(G)$.
This definition can be easily generalized to hypergraphs $H=(V,E)$. There we have $B_f(H)=\max\{|f(u)-f(v)|:\, \exists e \in E \text{ with } u,v\in e\}$.

Now we formally define the subject of our study. Let $k$ and $b$ be positive integers with $b \ge k-1$. For $A \subseteq [0,n]$ let $\underline{A}=\min(A)$ and $\overline{A}=\max(A)$. Further let $\binom{[0,n]}{k}=\{X \subseteq [0,n]: |X|=k\}$. Then $G_{n,k,b}$ is the graph with vertex set
\[
V_{n,k,b}=\left\{X \in \binom{[0,n]}{k}: \overline{X}- \underline{X}\leq b\right\}
\]
and edge set
\[
E_{n,k,b}=\left\{\{X,Y\} \in V_{n,k,b}: \overline{X\cup Y}- \underline{X\cup Y}\leq b\right\}.
\]
These graphs arise in the following context:
Let $G=(V,E)$ be a graph. A \emph{clique} is a subset of $V$ that induces a complete subgraph of $G$.
Consider the following transformation, which was used in the NP-completeness proof of the edge clique covering problem in \cite{NPc} by reducing it to the vertex clique covering problem. Let $V=[n]$. Then $\tilde{G}=(\tilde{V},\tilde{E})$ is constructed by setting $\tilde{V}=E$ and $\{\{i,j\},\{i',j'\}\}\in \tilde{E}$, if there is a clique $C\subseteq V$ with $i,j,i',j'\in C$. Let $\overline{\chi}_v(G)$ (resp. $\overline{\chi}_e(G)$) be the \emph{vertex clique covering number} (resp. \emph{edge clique covering number}) of the graph $G=(V,E)$, i.e.,  the minimal number of cliques whose induced subgraphs cover all vertices (resp. edges) of $G$. It can be shown, that $\overline{\chi}_e(G)=\overline{\chi}_v(\tilde{G})$, which is the essential part of the NP-completeness proof for the edge clique covering problem, since the transformation can be done in polynomial time. Here we want to generalize this result for hypergraphs $H=(V,E)$. We consider the \emph{2-section graph} of $H$, i.e., the graph $G_H=(V,E_H)$ on the vertex set of $H$ where $\{u,v\} \in E_H$ if there is an edge of $H$ containing $u$ and $v$. A subset $C$ of $V$ is called a \emph{weak clique} of $H$ if $C$ is a clique in $G_H$. A \emph{weak edge clique covering} of $H$ is a family ${\cal C}$ of weak cliques of $H$ such that for all $e\in E$ there is some $C\in {\cal C}$ with $e\subseteq C$. The \emph{weak edge clique covering number} of $H$ is the smallest size $\overline{\chi}_e(H)$ of a weak edge clique covering of $H$. It will turn out that the computation of $\overline{\chi}_e(H)$ can be simplified by the following transformation. The \emph{weak edge clique graph} of $H$ is the graph $\tilde{G}_H=(\tilde{V},\tilde{E})$ where $\tilde{V}=E$ and two vertices of $\tilde{G}_H$, i.e., edges $e,e'$ of $H$, are adjacent if there is a weak clique $C$ of $H$ containing $e$ and $e'$ as subsets. With these definitions we prove the following proposition.
\begin{prop}\label{prop_gamma}
Let $H=(V,E)$ be a hypergraph. Then
\[
\overline{\chi}_e(H)=\overline{\chi}_v(\tilde{G}_H).
\]
\end{prop}
\proof
Let ${\cal C}$ be a weak edge clique covering of $H$ of size $\overline{\chi}_e(H)$. For each $C\in {\cal C}$ let $\tilde{C}=\{e \in E: \, e\subseteq C\}$. Then $\tilde{C}$ is a clique in $\tilde{G}_H$ and ${\cal \tilde{C}}=\{\tilde{C}:\, C\in {\cal C}\}$ is a vertex clique covering of $\tilde{G}_H$. Consequently $\overline{\chi}_e(H)\geq \overline{\chi}_v(\tilde{G}_H)$.
Now let ${\cal \tilde{C}}$ be a vertex clique covering of $\tilde{G}_H$ of size $\overline{\chi}_v(\tilde{G}_H)$. For each $\tilde{C}\in {\cal \tilde{C}}$ let $C=\bigcup_{e\in \tilde{C}} e$. Then $C$ is a weak clique in $H$. To verify this fact, we pick two arbitrary vertices $x,y\in C$ and show that they are adjacent in $G_H$. First we consider the case that there is an edge $e\in \tilde{C}$ with $x,y\in e$. Then, by construction of $G_H$, $\{x,y\}\in E_H$. The alternative is, that there are 2 edges $e,e'\in \tilde{C}$ with $x\in e$ and $y\in e'$. Since $e$ and $e'$ are adjacent in $\tilde{G}_H$ there is a weak clique of $H$ containing both edges, which implies the adjacency of $x$ and $y$. Moreover, ${\cal \tilde{C}}$ is a weak edge clique covering of $H$. Consequently, $\overline{\chi}_v(\tilde{G}_H) \geq \overline{\chi}_e(H)$.
\qed

Thus, from an algorithmic point of view, it is enough to study the vertex clique covering problem. For bounded bandwidth, and more generally for bounded treewidth, there is a linear time dynamic programming algorithm for the solution \cite{CliqueCover2012}.
In an application, which will be described below, we were lead to the weak edge covering problem on a hypergraph whose bandwidth is small (and thus, theoretically, considered as bounded). This implies the following question: Given a hypergraph $H$ of bandwidth $b$, how large can be the bandwidth of the weak edge clique graph $\tilde{G}_H$ of $H$?
Here we discuss only $k$-uniform hypergraphs though many results can be simply generalized to hypergraphs whose edges have size at most $k$. For later computations, it is more suitable to take $[0,n]$ as the vertex set of $H$ instead of $[1,n]$. If, without loss of generality, $f(i)=i$ is the bandwidth numbering of $H$ then, obviously, $\tilde{G}_H$ has maximal bandwidth if $H$ contains all $k$-element subsets $X$ of $[0,n]$ with $\ol{X}-\ul{X} \le b$. In this case, $\tilde{G}_H$ is exactly the graph $\G$, which motivates the study of $\G$.

We came to these questions in the study of multielectrode recordings of neuronal signals, so-called spikes. Such recordings are carried out on multielectrode arrays, which can be used in-vivo or in-vitro. The denser the electrodes are placed the more likely it is for the neurons to be simultaneously recorded at different electrodes.
The resulting similarities in the recordings of the electrodes can provide useful information. In \cite{EnHa1} we developed an algorithm to estimate the (unknown) neighborhood of a neuron, i.e., the set of electrodes which record the signals of this neuron.
Such neighborhood information is also used as an additional tool in \cite{Prentice2011} for the so called \emph{spike sorting}, which is an estimated assignment of the recorded signals to the neurons.

Fix a short time interval in which several electrodes record signals. We consider these electrodes as vertices of a graph, which we call \emph{similarity-graph} for the fixed time interval.
First we mention that some neurons may always spike simultaneously. We combine such a set of neurons to one (artificial) new neuron. It might be an accident that two electrodes record a signal at almost the same time, but the simultaneous recording can also be caused by the fact that one spiking neuron has contact to both electrodes. Thus we do not test only one short time interval but several such intervals. If there are sufficiently many simultaneous recordings of two (or $k$) fixed electrodes, one may expect that these recordings are indeed caused by only one neuron and thus we draw an edge (hyperedge) between the corresponding vertices in the similarity-graph.

By algorithmic reasons, it is easier to check only pairs of electrodes, see \cite{EnHa1}.
But, with some more effort, also $k$-element subsets of electrodes could be checked for similarities if $k$ is small. This leads to edges and hyperedges of electrodes.
If a spiking neuron has contact to an unknown set $S$ of electrodes, all edges between any two vertices of $S$ (all hyperedges of any $k$ vertices of $S$) are drawn in the similarity-graph. Though these edges may also be caused by different simultaneously spiking neurons having contact in each case to two (or $k$) neurons, it is more likely that only one neuron is the source. Such a neuron yields the edges of a weak clique in the similarity-graph. Once the similarity-graph is constructed, it remains the question what is the basic cause for this graph. A reasonable answer is that as few as possible neurons yield the graph. Consequently, a minimum weak edge clique covering has to be determined. Because of the bounded length of the axons, only nearby electrodes, which are placed in form of a two-dimensional bounded grid (or some similar variants), may have contact to the same neuron. Hence the similarity-graph is a relatively sparse graph and edges are only drawn between electrodes which have a small Euclidean distance. Thus it is reasonable to expect that also this graph has a small bandwidth.

\section{Some basic properties}
\label{170}
\begin{lem}
\label{164}
Let $X$ and $Y$ be two distinct vertices of $G_{n,k,b}$. They are adjacent iff $\ol{X}-\ul{Y}\le b$ and $\ol{Y}-\ul{X} \le b$.
\end{lem}
\proof Let $X$ and $Y$ be adjacent. Then $\ol{X}-\ul{Y}\le \ol{X \cup Y} - \ul{X \cup Y}\le b$ and, analogously, $\ol{Y}-\ul{X}\le b$.

Now let $\ol{X}-\ul{Y}\le b$ and $\ol{Y}-\ul{X} \le b$. Then $\ol{X \cup Y} - \ul{X \cup Y} = \max\{\ol{X},\ol{Y}\}-\min\{\ul{X},\ul{Y}\} = \max\{\ol{X}-\ul{X},
\ol{X}-\ul{Y}, \ol{Y}-\ul{X}, \ol{Y}-\ul{Y}\} \le b.$
\qed

Note that $[i,i+k-1] \in \V$ iff $0 \le i \le n-(k-1)$ and that for $0 \le i < j  \le n-(k-1)$ the vertices $[i,i+k-1]$ and $[j,j+k-1]$ are adjacent iff $j \le i+b-(k-1)$.

\begin{lem}
\label{175}
Let $0 \le i < j  \le n-(k-1)$. The vertices $[i,i+k-1]$ and $[j,j+k-1]$ have distance $\lceil \frac{j-i}{b-k+1} \rceil$ in $\G$.
\end{lem}
\proof
Let $j-i=q(b-k+1)+r$ where $q$ is an integer and $1 \le r \le b-k+1$. Then $\lceil \frac{j-i}{b-k+1} \rceil=q+1$.
Obviously, the vertices $[i,i+(k-1)], [i+b-(k-1),i+b], [i+2b-2(k-1),i+2b-(k-1)], \dots, [i+qb-q(k-1),i+qb-(q-1)(k-1)], [j,j+(k-1)]$ form a path in $\G$ of length $q+1$.
Thus the distance is at most $q+1$.

If the vertices $X_0=[i,i+(k-1)], X_1, \dots, X_{l-1}, [j,j+(k-1)]=X_l$ form any path of length $l$ in $\G$, then, for $t=1,\dots,l$,
$\ol{X}_t-\ul{X}_{t-1}\le b$ and $\ol{X}_t-\ul{X}_t\ge k-1$, which implies
\begin{equation}
\label{183}
\ul{X}_t-\ul{X}_{t-1} \le b-(k-1).
\end{equation}
Summing up the inequalities (\ref{183}) for $t=1,\dots,l$ yields
$j-i \le l(b-k+1)$.
Since $j-i>q(b-k+1)$ we have $l \ge q+1$ and thus the distance is at least $q+1$.
\qed
\begin{corollary}
\label{193}
Let $X,Y \in \V$ and let $\ul{X} < \ul{Y}$ or $\ul{X} = \ul{Y}$ as well as $\ol{X} < \ol{Y}$.
Then $X$ and $Y$ have distance at most $\lceil \frac{\ol{Y}-\ul{X}-b}{b-k+1} \rceil+1$.
\end{corollary}
\proof
Let $i=\ul{X}$ and $j=\ol{Y}$.

{\bf Case 1.} $j-i \le b$. Then $\ol{Y}-\ul{X} \le b$ and $\ol{X}-\ul{Y} \le \ol{X} - \ul{X} \le b$. Consequently, $X$ and $Y$ are adjacent by Lemma \ref{164} and their distance is 1. Indeed, from the conditions on $X$ and $Y$ it follows that $\ol{Y} - \ul{X} > k-1$ and thus $\ol{Y} - \ul{X} -b > (-1)(b-k+1)$, which implies
$\lceil \frac{\ol{Y}-\ul{X}-b}{b-k+1} \rceil+1=1$.

{\bf Case 2.} $j-i > b$. Obviously, $X$ is adjacent to $X_1=[i+b-k+1,i+b]$ and $Y$ is adjacent to $Y_1=[j-b,j-b+k-1]$.

{\bf Case 2.1} $i+b-k+1 \ge j-b$. Then $X_1$ and $Y$ are adjacent since $\ol{X}_1-\ul{Y}=i+b-\ul{Y}=\ul{X}-\ul{Y}+b \le b$ and $\ol{Y}-\ul{X}_1=j-i-b+k-1 \le b$. Thus $X$ and $Y$ have distance at most 2. Indeed, $2 \le \lceil \frac{\ol{Y}-\ul{X}-b}{b-k+1} \rceil+1 \le \lceil \frac{b+(b-k+1)-b}{b-k+1} \rceil+1 =2$.

{\bf Case 2.2} $i+b-k+1 < j-b$. By Lemma \ref{175}, $X_1$ and $Y_1$ have distance $\lceil \frac{j-b-i-(b-k+1)}{b-k+1} \rceil=\lceil \frac{\ol{Y}-\ul{X}-b}{b-k+1} \rceil -1$, and thus $X$ and $Y$ have distance at most $\lceil \frac{\ol{Y}-\ul{X}-b}{b-k+1} \rceil+1.$
\qed

\begin{lem}
\label{338}
The graph $\G$ has the following number of vertices:
\[
|V_{n,k,b}|=(n-b+1)\binom{b}{k-1}+\binom{b}{k}=(n+1)\binom{b}{k-1}-(k-1)\binom{b+1}{k}.
\]
\end{lem}
\proof
The proof follows directly from the partition
\[
V_{n,k,b}=\bigcupdot_{i=0}^{n-b}\left\{ X\in \binom{[i,i+b]}{k}: \, \underline{X}=i\right\}\cupdot \left\{  X\in \binom{[n-b+1,n]}{k}\right\}.
\]
\qed

\section{Bandwidth for $b\geq \frac{n+k-1}{2}$}\label{sec3}
In the following, we often write the elements of $V_{n,k,b}$ as $k$-tuples in ascending order, i.e., $X=(i_1,\dots , i_k)$ with $i_1<\dots < i_k$.
Then $\underline{X}=i_1$ and $\overline{X}=i_k$. Furthermore, let $\overleftarrow{X}=(i_k,\dots , i_1)$ as well as $X^c=(n-i_k,\dots ,n-i_1)$. We collect all vertices that are adjacent to all other vertices in the set
\[
	C=\{ X\in \V: \, n-b\leq \underline{X} \le \overline{X} \leq b\}.
\]
Note that $C\neq \emptyset$ iff $b\geq \frac{n+k-1}{2}$ and that
\begin{align}\label{anzC}
|C|=\binom{2b-n+1}{k}.
\end{align}
We denote the set of remaining vertices by $R=\V\setminus C$, and split it into two parts:
\begin{align*}
R'&=\{X\in R: \, \underline{X}+\overline{X}\neq n \}, \\
R''&=\{X\in R: \, \underline{X}+\overline{X}= n \}.
\end{align*}
Let $R''=R''_0 \cupdot R''_1$ be a partition of $R''$ such that
\begin{align}\label{eqneareq}
||R''_0|-|R''_1||\leq 1.
\end{align}
We define a partition $R'=R'_0 \cupdot R'_1$ of $R'$ by
\begin{align*}
R'_0&=\{X\in R': \, \underline{X}+\overline{X}<n \},\\
R'_1&=\{X\in R': \, \underline{X}+\overline{X}>n \}
\end{align*}
and with $R_0=R'_0 \cupdot R''_0$ and $R_1=R'_1 \cupdot R''_1$ we have a partition $R=R_0 \cupdot R_1$.

\begin{lem}\label{lem1}
We have $||R_0|-|R_1||\leq 1$.
\end{lem}
\proof
A bijection between $R'_0$ and $R'_1$ is given by $X\mapsto X^c$. Hence we have $|R'_0|=|R'_1|$ and with \eqref{eqneareq} we obtain the assertion.
\qed

Recall the definition of the \emph{lexicographic ordering} $<_{lex}$ on the set of all $k$-tuples of integers:
\begin{align}\label{lexO}
(x_1,x_2,\dots , x_k)<_{lex} (y_1,y_2,\dots , y_k) \text{ if } \exists i\in [k] (\forall j\in [i-1]: \, x_j=y_j)\wedge x_i<y_i.
\end{align}
We define a proper numbering of $\V$ in the form of a total order $\le$. The minimal element gets label $1$, the next elements get labels $2,3,\dots$ and the maximal element gets label $|\V|$. Each total order will be given in the form of an ordinal sum of suborders:
If $\V=S_1 \cupdot \dots \cupdot S_l$ and $\le_i$ is a total order on $S_i$, $i=1,\dots,l$, then $\V=S_1 \oplus \dots \oplus S_l$ means that the elements of $\V$ are totally ordered as follows:
$X \le Y$ if there is some $i$ with $X,Y \in S_i$ and $X \le_i Y$ or there are some $i,j$ with $i < j$ and $X \in S_i$ and $Y \in S_j$.
We have $\V=R_0\cupdot C \cupdot R_1$, i.e., $l=3$ with $S_1=R_0$, $S_2=C$ and $S_3=R_1$. We define a total order $\le_{spo}$, which we call the \emph{simple palindrom ordering} (SPO), as follows:
\begin{align*}
\V=R_0\oplus C \oplus R_1,
\end{align*}
with the following suborders:
\begin{enumerate}
\item For all $X,Y\in R_0$: $X\leq_{spo}Y$ if $X\leq_{lex}Y$.
\item For all $X,Y\in C$: $X\leq_{spo}Y$ if $X\leq_{lex}Y$.
\item For all $X,Y\in R_1$: $X\leq_{spo}Y$ if $\overleftarrow{X}\leq_{lex}\overleftarrow{Y}$.
\end{enumerate}
Let $f_{spo}(X)$ be the label of $X\in V_{n,k,b}$ in the SPO.
Recall that the \emph{$f_{spo}$-distance} of  $X,Y\in V_{n,k,b}$ is given by
\[
d_{f_{spo}}(X,Y)=|f_{spo}(X)-f_{spo}(Y)|.
\]

\begin{lem}\label{lem2}
Let $X,Y$ be two adjacent elements of $V_{n,k,b}$ with $X<_{spo} Y$ and maximal $f_{spo}$-distance, where in addition $\underline{X}$ is minimal or $\overline{Y}$ is maximal. Then $X=[0,k-1]$ and $Y=[b-k+1,b]$ or $X=[n-b,n-b+k-1]$ and $Y=[n-k+1,n]$.
\end{lem}
\proof
If $Y\in C$ then $X=[0,k-1]$ and $Y=[b-k+1,b]$ have maximal $f_{spo}$-distance. Analogously, if $X\in C$ then $X=[n-b,n-b+k-1]$ and $Y=[n-k+1,n]$ have maximal $f_{spo}$-distance. It is not possible that $X$ and $Y$ lie both in $R_0$ or both in $R_1$, because in these cases $Y$ could be replaced by $[b-k+1,b]$ and $X$ by $[n-b,n-b+k-1]$, respectively. Thus it remains the case that $X\in R_0$ and $Y\in R_1$. To reach a maximal $f_{spo}$-distance, the form $X=(\underline{X},\underline{X}+1,\dots , \underline{X}+k-1)$ and $Y=(\overline{Y}-k+1,\overline{Y}-k+2,\dots , \overline{Y})$ with $\overline{Y}-\underline{X}=b$ is necessary. So we have $X=[i,i+k-1]$ and $Y=[i+b-k+1,i+b]$ for an $i\in [0,n-b]$. To prove the assertion, it is sufficient to show the following:
\begin{enumerate}[label=\alph{enumi})]
\item For $1\leq i\leq \frac{n-b}{2}$ we have $d_{spo}([i-1,i+k-2],[i+b-k,i+b-1])\geq d_{spo}([i,i+k-1],[i+b-k+1,i+b])$.
\item For $\frac{n-b}{2}<i\leq n-b$ we have $d_{spo}([i-1,i+k-2],[i+b-k,i+b-1])\leq d_{spo}([i,i+k-1],[i+b-k+1,i+b])$.
\end{enumerate}
We note that b) follows from a) because of the symmetry of the ordering. To show a) we define $I_i=\{X\in V_{n,k,b}: [i,i+k-1]\leq_{spo} X \leq_{spo} [i+b-k+1,i+b] \}$. To prove the inequality it is enough to show that the mapping $X=(i_1,i_2,\dots , i_k)\mapsto (i_1-1,i_2-1,\dots , i_k-1)=\tilde{X}$ is an injection $\phi$ from $I_i$ to $I_{i-1}$.
The injectivity is clear. Thus it remains to show that
$\tilde{X}=\phi (X)\in I_{i-1}$
if $X\in I_i$.

{\bf Case 1.} $\tilde{X}\in C$.

This case is easy, because $C\subseteq I_j$ for all $j$.

{\bf Case 2.} $\tilde{X}\in R_0$.
\begin{enumerate}
\item If $X\in R_0$ then $[i,i+k-1]\leq_{lex} X$ and thus $[i-1,i+k-2] \leq_{lex} \tilde{X}$, which yields $\tilde{X}\in I_{i-1}$.
\item If $X\in C$ then $i\leq \frac{n-b}{2}<n-b\leq \underline{X}$ and thus $i-1<\underline{X}-1=\underline{\tilde{X}}$. This implies that $[i-1,i+k-2] <_{lex} \tilde{X}$, which yields $\tilde{X}\in I_{i-1}$.
\item If $X\in R_1$ then $\overline{X}\leq \overline{[i+b-k+1,i+b]}=i+b$ because of $X\leq_{spo}[i+b-k+1,i+b]$ and due to $\underline{X}+\overline{X}\geq n$ we have $\underline{X}\geq n-i-b$. This implies $\underline{\tilde{X}}=\underline{X}-1\geq n-i-b-1\geq i-1$ because of $i\leq \frac{n-b}{2}$. Hence $\tilde{X}\in I_{i-1}$.
\end{enumerate}

{\bf Case 3.} $\tilde{X}\in R_1$.

\begin{enumerate}
\item The case $X\in R_0$ is not possible because $\underline{\tilde{X}}+\overline{\tilde{X}}=\underline{X}+\overline{X}-2<n$, which contradicts $\tilde{X}\in R_1$.
\item Let $X\in C$. Then $\overline{X}\leq b$, which implies $\overline{\tilde{X}}\leq b-1\leq i+b-2<\overline{[i+b-k,i+b-1]}$ due to $i\geq 1$. Hence $\overleftarrow{\tilde{X}}<_{lex} \overleftarrow{[i+b-k,i+b-1]}$, which yields $\tilde{X}\in I_{i-1}$.
\item If $X\in R_1$ then $\overline{X}\leq i+b$ and thus $\overline{\tilde{X}}\leq i+b-1$. This implies $\overleftarrow{\tilde{X}}\leq_{lex} \overleftarrow{[i+b-k,i+b-1]}$ and hence $\tilde{X}\in I_{i-1}$.
\end{enumerate}
\qed

Now the bandwidth of $f_{spo}$ can be determined:
\begin{lem}\label{prop2}
We have $B_{f_{spo}}(G_{n,k,b})=\left\lceil\frac{|V_{n,k,b}|+|C|-2}{2}\right\rceil$.
\end{lem}
\proof
We have $f_{spo}([0,k-1])=1$ and $f_{spo}([n-k+1,n])=|V_{n,k,b}|$. If $f_{spo}([n-b,n-b+k-1])=u+1$, then $f_{spo}([b-k+1,b])=u+|C|$. Because of Lemma \ref{lem1} we have $|u-(|V_{n,k,b}|-|C|-u)|\leq 1$ and thus $|(u+|C|-1)-(|V_{n,k,b}|-u-1)|\leq 1$. Lemma \ref{lem2} implies that one of the $f_{spo}$-distances $d_{f_{spo}}([0,k-1],[b-k+1,b])=u+|C|-1$ and $d_{f_{spo}}([n-b,n-b+k-1],[n-k+1,n])=|V_{n,k,b}|-u-1$ is the maximal $f_{spo}$-distance. As they both differ from each other by at most $1$ it follows that
\[
B_{f_{spo}}(G_{n,k,b})=\left\lceil\frac{u+|C|-1+|V_{n,k,b}|-u-1}{2}\right\rceil=\left\lceil\frac{|V_{n,k,b}|+|C|-2}{2}\right\rceil.
\]
\qed

Now we are able to prove the first part of Theorem \ref{thm1}.
\begin{proof}[Proof of Theorem \ref{thm1}.a)]
We know from Lemma \ref{prop2} that $B(G_{n,k,b})\leq \left\lceil\frac{|V_{n,k,b}|+|C|-2}{2}\right\rceil$. Let $f$ be an arbitrary proper numbering of $G_{n,k,b}$. Let $X_V$ be the vertex with number $1$ and $X^V$ the vertex with number $|V_{n,k,b}|$. Further let
$X_C$ be the vertex of $C$ with smallest number, denoted $\alpha$,  and $X^C$ be the veretx of $C$ with largest number, denoted  $\beta$. Then $\beta - \alpha\geq |C|-1$. Further $X_V$ and $X^C$ as well as $X_C$ and $X^V$ are adjacent with $d_f(X_V,X^C)=\beta -1$ and $d_f(X_C,X^V)=|V_{n,k,b}|-\alpha$. The sum of them is
\[
s=(\beta -1)+(|V_{n,k,b}|-\alpha)=|V_{n,k,b}|+(\beta - \alpha)-1\geq |V_{n,k,b}|+|C|-2.
\]
The maximum of both $f$-distances is therefore at least $\left\lceil\frac{|V_{n,k,b}|+|C|-2}{2}\right\rceil$.
From Lemma \ref{338} and \eqref{anzC} we obtain
\[
B(G_{n,k,b})=  \left\lceil\frac{(n+1)\binom{b}{k-1}-(k-1)\binom{b+1}{k}+\binom{2b-n+1}{k}-2}{2}\right\rceil.
\]
\end{proof}

\section{Asymptotic bandwidth for $b=o(n)$}\label{sec4}
In this section, we consider the case, where $b$ grows sublinearly with respect to $n$. First we take a simple proper numbering, which provides an upper bound for the bandwidth.
\begin{lem}
\label{401}
Let $n,k,b$ be arbitrary integers with $1 \le k-1 \le b \le n$. Then
\[
B(G_{n,k,b})\leq k\binom{b}{k}.
\]
\end{lem}
\proof
We order the vertices of $G_{n,k,b}$ in a lexicographic way, see \eqref{lexO}. Let $f_{lex}(X)$ be the label of $X\in V_{n,k,b}$ with respect to this ordering. Now let $X$ and $Y$ be two adjacent vertices with $X<_{lex} Y$ and let $X'=[\underline{X},\underline{X}+k-1]$ and $Y'=[\underline{X}+b-k+1,\underline{X}+b]$. Then
\begin{align}\label{lexB}
X'\leq_{lex} X<_{lex} Y \leq_{lex} Y'.
\end{align}
Moreover, for $j \in [0,n-b]$,
\begin{align*}
|\{X\in V_{n,k,b}:\, \underline{X}=j\}|=\binom{b}{k-1}.
\end{align*}
and, for $j \in [n-b+1,n]$,
\begin{align*}
|\{X\in V_{n,k,b}:\, \underline{X}=j\}|\le\binom{b}{k-1}.
\end{align*}
Since $Y'$ is the lexicographically smallest vertex with minimum element $\underline{X}+b-k+1$ it follows that
\begin{align}
\label{700}
|f_{lex}(Y')-f_{lex}(X')|\leq (b-k+1)\binom{b}{k-1}=k\binom{b}{k}.
\end{align}
Now \eqref{lexB}  and \eqref{700} imply
\begin{align*}
|f_{lex}(Y)-f_{lex}(X)|\leq k\binom{b}{k},
\end{align*}
which proves the assertion.
\qed

Chv\'atal observed in \cite{bwdiamb} that a lower bound for the bandwidth is given by
\begin{align}\label{lbb}
B(G)\geq \left\lceil \frac{|V|-1}{\diam(G)}\right\rceil.
\end{align}
Here the \emph{diameter} $\diam(G)$ of the graph $G=(V,E)$ is the maximal distance of any two vertices of $G$.

By Corollary \ref{193}, the distance of any two vertices of $\G$ is at most $\lceil \frac{n-0-b}{b-k+1} \rceil + 1= \lceil \frac{n-k+1}{b-k+1} \rceil$ and by Lemma \ref{175} the vertices $[0,k-1]$ and $[n-k+1,n]$ have distance $\lceil \frac{n-k+1}{b-k+1} \rceil$. Accordingly,

\begin{align}
\label{diam_Gnkb}
\diam(\G)=\left\lceil \frac{n-k+1}{b-k+1}\right\rceil.
\end{align}
Now we have all preparations to prove Theorem \ref{thm1} b).
\begin{proof}[Proof of Theorem \ref{thm1} b)]
From Lemma \ref{401} we know that
\[
B(\G)\leq k\binom{b}{k}.
\]
For the lower bound, we use the fact that $\binom{b}{k}=\frac{b^k}{k!}+O(b^{k-1})$ as $b \to \infty$. We have by Lemma \ref{338},  \eqref{lbb} and  \eqref{diam_Gnkb} for $n\to \infty$:
\begin{align*}
B(\G)&\ge \left\lceil \frac{|\V|-1}{\diam(\G)}\right\rceil = \frac{(n-b+1)\binom{b}{k-1}+\binom{b}{k}-1}{\left\lceil \frac{n-k+1}{b-k+1}\right\rceil}=
\frac{n\binom{b}{k-1}+O(b^{k})}{\frac{n-k+1}{b-k+1}+O(1)}\\
&=\frac{nk\binom{b}{k}+O(b^{k+1})}{n+O(b)}=\frac{k\binom{b}{k}+O(\frac{b^{k+1}}{n})}{1+O(\frac{b}{n})}.
\end{align*}

In the case $b=o(n^{\frac{1}{k+1}})$  we have
\[
\frac{k\binom{b}{k}+O(\frac{b^{k+1}}{n})}{1+O(\frac{b}{n})} \ge \left(k\binom{b}{k} + o(1)\right)\left(1-O(\frac{b}{n})\right) = k\binom{b}{k} - o(1).
\]

Thus,  $B(\G)\geq k\binom{b}{k}$ for sufficiently large $n$, which proves the first part of the assertion.

If $b=o(n)$ then
\[
\frac{k\binom{b}{k}+O(\frac{b^{k+1}}{n})}{1+O(\frac{b}{n})}\sim k \binom{b}{k},
\]
which shows that $B(\G)\gtrsim k\binom{b}{k}$ as $n\to \infty$, which proves the second part of the assertion.
\end{proof}

\section{Further basic properties for the asymptotics}
\label{451}
\begin{lem}
\label{437}
Let $b \sim \beta n$, $\delta > 0$, $X,Y \in \V$ and $\ol{Y}-\ul{X} \lesssim i(1-\delta) \beta n$ as well as
$\ol{X}-\ul{Y} \lesssim i(1-\delta) \beta n$, where $n \rightarrow \infty$ and $i$ is a positive integer.
If $n$ is sufficiently large, then $X$ and $Y$ have distance at most $i$.
\end{lem}
\proof Without loss of generality, let $(\ul{X},\ol{X}) \le_{lex} (\ul{Y},\ol{Y})$, i.e., $\ul{X} < \ul{Y}$ or $\ul{X} = \ul{Y}$ as well as $\ol{X} < \ol{Y}$.
By Corollary \ref{193}, $X$ and $Y$ have distance at most $\lceil\frac{i(1-\delta) \beta n -b}{b-k+1}\rceil + 1=\lceil\frac{i(1-\delta) \beta -\beta + o(1)}{\beta + o(1)}\rceil + 1=\lceil i(1-\delta)-1+o(1)\rceil + 1 \le i$. \qed

Let $P$ be a polygon in $\Om=\{(x,y) \in \R^2: 0 \le x \le y \le 1\}$ and let $\inte(P)$ be the interior of $P$. Let
\begin{align*}
V_{n,k}(P)&=\left\{X \in \binom{[0,n]}{k}:\ \frac{1}{n}(\ul{X},\ol{X}) \in P\right\},\\
V_{n,k}^o(P)&=\left\{X \in \binom{[0,n]}{k}:\ \frac{1}{n}(\ul{X},\ol{X}) \in \inte(P)\right\}.
\end{align*}
\begin{lem}
We have
\[
|V_{n,k}^o(P)| \sim |V_{n,k}(P)| \sim \left(\frac{1}{(k-2)!}\iint_P (y-x)^{k-2}\,dx\,dy\right) n^k \text{ as } n \rightarrow \infty.
\]
\end{lem}
\proof Let $i,j$ be integers with $0 \le i \le j \le n$. Obviously,
\[
|\{X \in \binom{[0,n]}{k}:\ \ul{X}=i, \ol{X}=j\}|=\binom{j-i-1}{k-2}.
\]
Thus
\[
|V_{n,k}(P)|= \sum_{\frac{1}{n}(i,j)\in P} \binom{j-i-1}{k-2} = \sum_{\frac{1}{n}(i,j)\in P} \binom{n(j/n-i/n)-1}{k-2}.
\]
For $0 \le z \le 1$, we have
\[
\binom{nz-1}{k-2} \le \frac{z^{k-2}}{(k-2)!} n^{k-2}.
\]
Accordingly,
\[
\frac{1}{n^k}|V_{n,k}(P)| \le \frac{1}{(k-2)!} \sum_{\frac{1}{n}(i,j)\in P} \frac{1}{n^2}(j/n-i/n)^{k-2}.
\]
The RHS is a Riemann sum for the integral $\iint_P (y-x)^{k-2}\,dx\,dy$, which shows that
\[
|V_{n,k}(P)| \lesssim \left(\frac{1}{(k-2)!}\iint_P (y-x)^{k-2}\,dx\,dy\right) n^k \text{ as } n \rightarrow \infty.
\]
For $\delta > 0$ let $P_{\delta}=\{(x,y) \in P: y-x \ge \delta\}$.
Clearly, $P_{\delta} \subseteq P$. Obviously, for any $\varepsilon > 0$ there is some $\delta >0$ such that
\[
\iint_{P_{\delta}} (y-x)^{k-2}\,dx\,dy \ge (1-\varepsilon) \iint_P (y-x)^{k-2}\,dx\,dy.
\]
Moreover, for any $z \ge \delta$ there is some $n_0$ such that for all $n > n_0$
\[
\frac{k-2}{nz} \le \varepsilon.
\]
This implies
\[
(1-\varepsilon)^{k-2}\frac{(nz)^{k-2}}{(k-2)!} \le
\left(1-\frac{k-2}{nz}\right)^{k-2}\frac{(nz)^{k-2}}{(k-2)!} \le
\binom{nz-1}{k-2}
\]
and further
\begin{align*}
\frac{1}{n^k}|V_{n,k}(P)| &\ge (1-\varepsilon)^{k-2} \frac{1}{(k-2)!} \sum_{\frac{1}{n}(i,j)\in P_{\delta}} \frac{1}{n^2}(j/n-i/n)^{k-2} \\
&\gtrsim (1-\varepsilon)^{k-1} \frac{1}{(k-2)!} \iint_P (y-x)^{k-2}\,dx\,dy.
\end{align*}
Now, with $\varepsilon \rightarrow 0$ we obtain
\[
|V_{n,k}(P)| \gtrsim \left(\frac{1}{(k-2)!} \iint_P (y-x)^{k-2}\,dx\,dy \right) n^k \text{ as } n \rightarrow \infty.
\]
The reasoning for $|V_{n,k}^o(P)|$ is the same. \qed

For the sake of brevity, we define the \emph{measure of the polygon} $P \subseteq \Om$ by
\[
\mu(P)=\frac{1}{(k-2)!} \iint_P (y-x)^{k-2}\,dx\,dy.
\]
\begin{corollary}
\label{cor2}
If $S \subseteq \binom{[0,n]}{k}$ is a family of sets that contains all $X \in \binom{[0,n]}{k}$ with
$\frac{1}{n}(\ul{X},\ol{X}) \in \inte(P)$ and some $X \in \binom{[0,n]}{k}$ with
$\frac{1}{n}(\ul{X},\ol{X})$ on the boundary of $P$, then
\[
|S| \sim \mu(P) n^k \text{ as } n \rightarrow \infty.
\]
\end{corollary}

\section{Definition and measure of crucial polygons}
Recall that we consider the case $b\sim \beta n$ with $\beta \in (0,0.5]$ and $q\in \N$ and $r\in \R$ such that $1=q\beta +r$.
We define in $\Om$ several sets of points. First let for $i=1,\dots,q$
\begin{align*}
A_i&=(ir,ir+q(\beta-r)),\\
B_i&=(r-\beta+i\beta, r-\beta+i\beta),\\
C_i&=(i\beta,i\beta).
\end{align*}
In the following we denote lines given by $y=ax+b$ (or, more generally, by $ax+by=c$) by $g_{y=ax+b}$ (or $g_{ax+by=c}$).
Note that the points $B_i$ and $C_i$ lie on $g_{y=x}$. Furthermore, the points $A_i$ lie above or on the line $g_{y=x+\beta}$ iff $r \le \frac{q-1}{q^2+q-1}$. This is the reason for the distinction between a) and b) in Theorem \ref{thm2}.

For the Case a), i.e., for $r \le \frac{q-1}{q^2+q-1}$, we define points $D_i$ and $E_i$ as the intersection points of the segments
$\ol{A_iB_i}$ resp. $\ol{A_iC_i}$ with the line $g_{y=x+\beta}$.
An easy computation yields with
\[
\gamma=\beta(1-1/q)
\]
that for $i=1,\dots,q$
\begin{align*}
D_i&=(r+(i-1) \gamma, r+(i-1)\gamma+\beta),\\
E_i&=(i\gamma,i\gamma+\beta).
\end{align*}
For the Case b), i.e., for $r > \frac{q-1}{q^2+q-1}$, we define points $F_i$ as those points that provide an equipartition into $q+1$ parts of the segment between $(0,0)$ and $(1,1)$ and points $G_i$ as the intersection points of the segments between $F_i$ and $(0,1)$ with the line $g_{y=x+\beta}$. It is straightforward that for $i=0,\dots,q+1$
\begin{align*}
F_i&=(i/(q+1),i/(q+1)),\\
G_i&=(i(1-\beta)/(q+1),i(1-\beta)/(q+1)+\beta).
\end{align*}

As for the Case a), let $E_q$ be the intersection point of the lines $g_{y=x+\beta}$ and $g_{y=q\beta}$.
Finally, we define an auxiliary point $H_1=(r,\beta)$. For the Case a) we also use the points $C_0=F_0=(0,0)$, $B_{q+1}=F_{q+1}=(1,1)$, $E_0=G_0=(0,\beta)$ and $D_{q+1}=G_{q+1}=(1-\beta,1)$.

In Figures \ref{fig:points a} and \ref{fig:points b} the points are illustrated for
$\beta=9/20$ (i.e., Case a) with $q=2, r=1/10 < 1/5$) as well as for $\beta=7/20$ (i.e., Case b) with $q=2, r=3/10 > 1/5$).

\begin{figure}[H]
  \centering
   \begin{tikzpicture}[scale=1.5]
   \foreach \s in {5} {
   \foreach \b in {0.45}{
   \foreach \q in {2} {
   \foreach \r in {0.1} {
   \foreach \g in {\b*(1-1/\q)} {

   \coordinate[label=below:{$C_0=F_0=(0,0)$}] (0) at (0,0);
   \coordinate[label=below:{$(1,0)$}] (10) at (\s,0);
   \coordinate[label=above:{$(0,1)$}] (01) at (0,\s);
   \coordinate[label=above:{$B_3=F_3=(1,1)$}] (11) at (\s,\s);
   \draw[] (0) -- (10) -- (11) -- (01) -- cycle;

   \foreach \i in {1,...,\q} {
   \coordinate[label=above:{$A_{\i}$}] (A\i) at (\s*\i*\r,{\s*(\i*\r+\q*(\b-\r))});
   \coordinate[label=below:{$B_{\i}$}] (B\i) at ({\s*(\r-\b+\i*\b)},{\s*(\r-\b+\i*\b)});
   \coordinate[label=below:{$C_{\i}$}] (C\i) at (\s*\i*\b,\s*\i*\b);
   \coordinate[label=right:{$D_{\i}$}] (D\i) at ({\s*(\r+(\i-1)*\g)},{\s*(\r+(\i-1)*\g+\b)});
   \coordinate[label=below:{$E_{\i}$}] (E\i) at ({\s*(\i*\g)},{\s*(\i*\g+\b)});

   \draw (A\i) -- (B\i);
   \draw (A\i) -- (C\i);
   };
   \coordinate[label=left:{$E_0=G_0$}] (E0) at (0,\s*\b);
   \coordinate[label=above:{$D_3=G_3$}] (D3) at ({\s*(1-\b)},\s);
   \foreach \i in {A,B,C,D,E} {
   \foreach \j in {1,...,\q} {
   \fill (\i\j) circle(1pt);

   };
   };
   };
   };
   };
   };
   \fill (E0) circle(1pt);
   \fill (D3) circle(1pt);
   \fill (0) circle(1pt);
   \fill (10) circle(1pt);
   \fill (01) circle(1pt);
   \fill (11) circle(1pt);
   \draw (0) -- (11);
   \draw (E0) -- (D3);

 };

\end{tikzpicture}
    \caption{Polygons and important points for $\beta=9/20$ (i.e., Case a) with $q=2, r=1/10 < 1/5$).}		
    \label{fig:points a}
\end{figure}
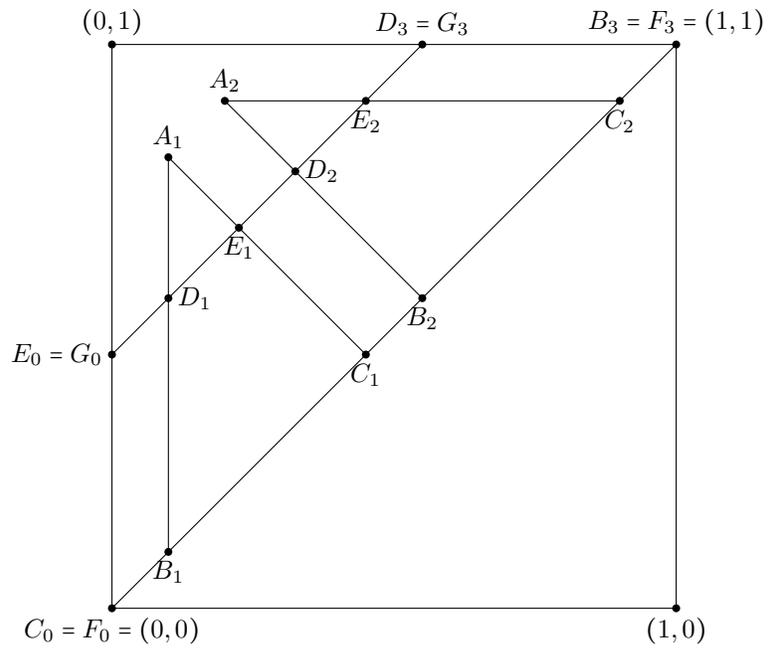	
	
\begin{figure}[H]		
\centering
 \begin{tikzpicture}[scale=1.5]
   \foreach \s in {5} {
   \foreach \b in {0.35}{
   \foreach \q in {2} {
   \foreach \r in {0.3} {
   \foreach \g in {\b*(1-1/\q)} {

   \coordinate[label=below:{$C_{0}=F_0=(0,0)$}] (0) at (0,0);
   \coordinate[label=below:{$(1,0)$}] (10) at (\s,0);
   \coordinate[label=above:{$(0,1)$}] (01) at (0,\s);
   \coordinate[label=above:{$B_3=F_3=(1,1)$}] (11) at (\s,\s);
   \draw[] (0) -- (10) -- (11) -- (01) -- cycle;

   \foreach \i in {1,...,\q} {
   \coordinate[label=45:{$A_{\i}$}] (A\i) at (\s*\i*\r,{\s*(\i*\r+\q*(\b-\r))});
   \coordinate[label=left:{$B_{\i}$}] (B\i) at ({\s*(\r-\b+\i*\b)},{\s*(\r-\b+\i*\b)});
   \coordinate[label=right:{$C_{\i}$}] (C\i) at (\s*\i*\b,\s*\i*\b);
  \coordinate[label=315:{$F_{\i}$}] (F\i) at ({\s*(\i/(\q+1))},{\s*(\i/(\q+1))});
  \coordinate[label=left:{$G_{\i}$}] (G\i) at ({\s*(\i*(1-\b)/(\q+1))},{\s*(\i*(1-\b)/(\q+1)+\b)});

   \draw (A\i) -- (B\i);
   \draw (A\i) -- (C\i);
   \draw (01) -- (F\i);
   };
   \coordinate[label=left:{$G_{0}$}] (G0) at (0,\s*\b);
   \coordinate[label=above:{$G_{3}$}] (G3) at ({\s*(1-\b)},\s);
   \coordinate[label=left:{$H_1$}] (H1) at (\s*\r,\s*\b);
   \coordinate[label=below:{$E_{2}$}] (E2) at ({\s*(2*\g)},{\s*(2*\g+\b)});
   \foreach \i in {A,B,C,F,G} {
   \foreach \j in {1,...,\q} {
   \fill (\i\j) circle(1pt);

   };
   };
   };
   };
   };
   };
   \fill (G0) circle(1pt);
   \fill (H1) circle(1pt);
   \fill (G3) circle(1pt);
   \fill (E2) circle(1pt);
   \fill (0) circle(1pt);
   \fill (10) circle(1pt);
   \fill (01) circle(1pt);
   \fill (11) circle(1pt);
   \draw (0) -- (11);
   \draw (G0) -- (G3);
   \draw (H1) -- (C1);
   \draw (E2) -- (A2);

 };
\end{tikzpicture}
    \caption{Polygons and important points for $\beta=7/20$ (i.e., Case b) with $q=2, r=3/10 > 1/5$).}
    \label{fig:points b}
\end{figure}
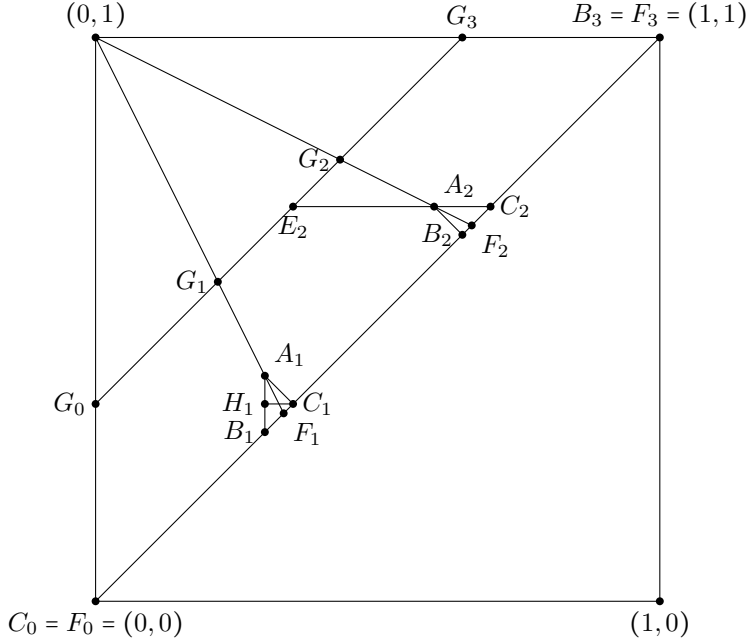

It is easy to check that the points $A_i$ lie on the segments $\ol{F_iG_i}$ (we have $\overrightarrow{F_iA_i}=(1-r(q+1)/\beta) \overrightarrow{F_iG_i}$) and that the points $F_0,B_1,F_1,C_1,\dots, B_q,F_q,C_q,F_{q+1}$ lie in this order on the line $g_{y=x}$.
It is also easy to see that $H_1$ lies on the segment $\ol{A_1B_1}$.

\begin{lem}
Let $0 \le s < t \le 1$, let $P_1=(\xi_1,\xi_1+s),P_2=(\xi_2,\xi_2+s)$ be points on the line $g_{y=x+s}$ with $0 \le \xi_1 \le \xi_2 \le 1-s$
and let $P_3=(\xi_3,\xi_3+t),P_4=(\xi_4,\xi_4+t)$ be points on the line $g_{y=x+t}$ with $0 \le \xi_3 \le \xi_4 \le 1-t$.
Let $u=\xi_2-\xi_1$ and $v=\xi_4-\xi_3$. Then
\[
\mu(P_1P_2P_4P_3) = \frac{1}{(k-2)!} \frac{1}{t-s} \left( \frac{1}{k}(v-u)(t^k-s^k)+\frac{tu-sv}{k-1}(t^{k-1}-s^{k-1})\right).
\]
\end{lem}

\proof The proof follows directly by computing the integrals, using the coordinate transformation $x'=x+y$, $y'=y-x$. This leads to a domain of integration in form of a trapezoid whose basis is parallel to the $x'$-axis. \qed

Since only the difference of the $\xi$-values has influence we have:

\begin{corollary}
\label{883}
Let $0 \le s < t \le 1$, let $P_1,P_2$ and $Q_1,Q_2$ be points of $\Omega$ on the line $g_{y=x+s}$
with $\overrightarrow{P_1P_2} = \overrightarrow{Q_1Q_2}$ showing to north east and let
$P_3,P_4$ and $Q_3,Q_4$ be points on the line $g_{y=x+t}$
with $\overrightarrow{P_3P_4} = \overrightarrow{Q_3Q_4}$ showing to north east.
Then
\[
\mu(P_1P_2P_4P_3) = \mu(Q_1Q_2Q_4Q_3).
\]
\end{corollary}

By inserting the corresponding values and using the definition of the $c$-functions in Theorem \ref{thm2}, we obtain:

\begin{corollary}
\label{574}
We have for all possible $i$:
\begin{align*}
\mu(F_0F_{q+1}G_{q+1}G_0)&=\frac{\beta^{k-1}}{k!}(k-(k-1)\beta) = (q+1)c_2(\beta,k),\\
\mu(C_iB_{i+1}D_{i+1}E_i)&=\frac{\beta^{k-1}}{(k-1)!}r = (q+1)c_2(\beta,k)-qc_1(\beta,k),\\
\mu(B_iC_iE_iD_i)&=\frac{\beta^{k-1}}{k!}\left(k(\beta-r)-\frac{\beta(k-1)}{q}\right)
\\
&=(q+1)(c_1(\beta,k)-c_2(\beta,k)),\\
\mu(F_iF_{i+1}G_{i+1}G_i)&=\frac{1}{q+1} \mu(F_0F_{q+1}G_{q+1}G_0)=c_2(\beta,k),\\
\mu(A_iB_iC_i)&=\frac{(\beta-r)^k}{k!}q^{k-1}=(q+1)c_3(\beta,k),\\
\mu(A_iB_iF_i)&=(1-\frac{i}{q+1})\mu(A_iB_iC_i)=(q+1-i)c_3(\beta,k),\\
\mu(A_iF_iC_i)&=\frac{i}{q+1} \mu(A_iB_iC_i)=ic_3(\beta,k),\\
\mu(B_1C_1H_1)&=\frac{1}{q^{k-1}} \mu(A_1B_1C_1)=\frac{q+1}{q^{k-1}}c_3(\beta,k).
\end{align*}
\end{corollary}

\section{Proof of the lower bounds for the bandwidth in Theorem \ref{thm2}}

We fix some $\varepsilon > 0$, where $\varepsilon$ is sufficiently small, in particular $\varepsilon < 1$.
Let $P$ be a polygon in $\Om$. In the following we work with dilations of $P$ around a given centerpoint by factors of the form $(1-\varepsilon)$ and $(1+\varepsilon)$, respectively. We denote the new polygons by $\ul{P}$ and $\ol{P}$, respectively. But we emphasize that $\ul{P}$ and $\ol{P}$ depend on $\varepsilon$ and on the centerpoint of dilation. Note that
\[
\lim_{\varepsilon \rightarrow 0} \mu(\ul{P}) = \lim_{\varepsilon \rightarrow 0} \mu(\ol{P}) = \mu(P).
\]
Though the following result follows also directly from Lemma \ref{338}, we prove it as an example for our
polygon-method:
\begin{lem}
\label{628}
We have
\[
|\V| \sim \mu(F_0F_{q+1}G_{q+1}G_0) n^k = (q+1) c_2(\beta,k) n^k.
\]
\end{lem}
\proof Let briefly $T=F_0F_{q+1}G_{q+1}G_0$. Note that $T=\{(x,y) \in \Om: y \le x+ \beta\}.$
We choose $F_0$ as the centerpoint of dilation so that e.g. the vertices of $\ul{T}$ are given by $\ul{F}_0=(0,0),
\ul{F}_{q+1}=(1-\varepsilon)(1,1), \ul{G}_{q+1}=(1-\varepsilon)(1-\beta,1), \ul{G}_{0}=(1-\varepsilon)(0,\beta)$.
If the dilation factors are $(1-\varepsilon)$ and $(1+\varepsilon)$, respectively, then
$\ul{T} \subseteq \{(x,y) \in \Om: y \le x+(1-\varepsilon)\beta\}$ and
$\ol{T}\cap \Om \subseteq \{(x,y) \in \Om: y \le x+(1+\varepsilon)\beta\}$.
It is easy to check that for sufficiently large $n$ and any $X \in \binom{[0,n]}{k}$ the following implications are true:
\[
\frac{1}{n}(\ul{X},\ol{X}) \in \ul{T} \Rightarrow X \in \V \Rightarrow \frac{1}{n}(\ul{X},\ol{X}) \in \ol{T}\cap \Om.
\]
Thus
\[
|V_{n,k}(\ul{T})| \le |\V| \le |V_{n,k}(\ol{T} \cap \Om)|
\]
and hence by Corollary \ref{cor2}
\[
\mu(\ul{T})n^k \lesssim |\V| \lesssim \mu(\ol{T})n^k.
\]
With $\varepsilon \rightarrow 0$ we obtain (using also Corollary \ref{574})
\[
|\V| \sim \mu(T) n^k=(q+1) c_2(\beta,k) n^k.
\]
\qed

Now we prove the first asymptotic lower bound:
\begin{lem}
We have
\[
B(\G) \gtrsim \frac{1}{q} \mu(F_0C_qE_qG_0) n^k = c_1(\beta,k) n^k \text{ as } n \rightarrow \infty.
\]
\end{lem}
\proof Let briefly $Q=F_0C_qE_qG_0$. For the dilation, we choose again $F_0$ as the centerpoint. With the factor $(1-\varepsilon)$ we obtain $\ul{Q}$. Let $G'$ be the subgraph of $\G$ induced by
\[
V'=\left\{X  \in \V: \frac{1}{n}(\ul{X},\ol{X}) \in \ul{Q}\right\}.
\]
Note that $\frac{1}{n}\ol{X} \le q\beta(1-\varepsilon)$ for all $X \in V'$.

Now we show that for sufficiently large $n$
\begin{equation}
\label{668}
\diam(G') \le q.
\end{equation}
Let $X$ and $Y$ be any two distinct vertices of $G'$.
Then $0 \le \ul{X} \le \ol{X} \le q \beta (1-\varepsilon) n$ and $0 \le \ul{Y} \le \ol{Y} \le q \beta (1-\varepsilon) n$. By Lemma \ref{437}, $X$ and $Y$ have distance at most $q$ if $n$ is sufficiently large,
which proves \eqref{668}.
From the Chv\'atal bound \eqref{lbb} it follows that
\[
B(\G) \ge B(G') \ge \frac{|V'|-1}{q} \gtrsim \frac{\mu(\ul{Q})}{q} n^k
\]
and with $\varepsilon \rightarrow 0$
\[
B(\G)  \gtrsim \frac{\mu(Q)}{q} n^k.
\]
Let briefly $Q_q=C_qB_{q+1}D_{q+1}E_q = C_qF_{q+1}G_{q+1}E_q$ and recall $T=F_0F_{q+1}G_{q+1}G_0$. Then $Q=T \setminus Q_q$ and hence
\[
\mu(Q)=\mu(T)-\mu(Q_q).
\]
By Corollary \ref{574},
\[
\mu(Q)=(q+1)c_2(\beta,k)-\left((q+1)c_2(\beta,k)-qc_1(\beta,k)\right) = q c_1(\beta,k).
\]
\qed

Now we prove the second asymptotic lower bound, which applies only for the second case:
\begin{lem}
If $r > \frac{q-1}{q^2+q-1}$ then
\begin{align*}
B(\G) &\gtrsim \left(\mu(F_0F_1G_1G_0)+\frac{1}{q+1}\mu(B_1C_1H_1)\right) n^k\\
&=\left(c_2(\beta,k)+\frac{1}{q^{k-1}}c_3(\beta,k)\right) n^k.
\end{align*}
\end{lem}
\proof Let briefly $R=B_1C_1H_1$. We choose $\frac{1}{2}(B_1+C_1)=\frac{1}{2}(r+\beta,r+\beta)$ as the centerpoint and $(1-\frac{2}{\beta-r}\varepsilon)$ as the factor of dilation and thus obtain $\ul{R}$ from $R$. The vertices of $\ul{R}$ are
$\ul{B}_1=(r+\varepsilon, r+\varepsilon)$, $\ul{C}_1=(\beta-\varepsilon,\beta-\varepsilon)$ and $\ul{H}_1=(r+\varepsilon, \beta - \varepsilon)$.

Let $f$ be a bandwidth numbering of $\G$. Let $X_V$ and $X^V$ be those vertices for which $f(X_V)=1$ and $f(X^V)=|\V|$.

{\bf Case 1.} $\ul{X}_V \le n(r+\varepsilon)$ and $\ul{X}^V \le n(r+\varepsilon)$.

Then, for sufficiently small $\varepsilon$ and sufficiently large $n$, $\ol{X}_V,\ol{X}^V \le n(r+\varepsilon+\beta+\varepsilon) \le 2\beta(1-\varepsilon) n$. Lemma \ref{437} implies that $X_V$ and $X^V$ have distance at most $2$
and by
Lemma \ref{628}, we have
\[
B(\G) \ge \frac{|\V|-1}{2} \gtrsim \frac{q+1}{2}c_2(\beta,k) n^k \gtrsim \left(c_2(\beta,k)+\frac{1}{q^{k-1}}c_3(\beta,k)\right) n^k.
\]

{\bf Case 2.} $\ul{X}_V > n(r+\varepsilon)$ and $\ul{X}^V > n(r+\varepsilon)$.

Then, for sufficiently large $n$, $\ol{X}^V-\ol{X}_V,\ol{X}_V-\ol{X}^V \le n-n(r+\varepsilon) = (1-r-\varepsilon)n = q\beta(1-\frac{\varepsilon}{q\beta}) n$.
Lemma \ref{437} implies that $X_V$ and $X^V$ have distance at most $q$
and by
Lemma \ref{628}, we have

\[
B(\G) \ge \frac{|\V|-1}{q} \gtrsim \frac{q+1}{q}c_2(\beta,k)n^k \gtrsim  \left(c_2(\beta,k)+\frac{1}{q^{k-1}}c_3(\beta,k)\right) n^k.
\]

{\bf Case 3.} $\ul{X}_V \le n(r+\varepsilon)$ and $\ul{X}^V > n(r+\varepsilon)$.

Let $X_R=\argmin\{f(X): X \in V_{n,k}(\ul{R})\}$ and $X^R=\argmax\{f(X): X \in V_{n,k}(\ul{R})\}$.
Clearly,
\begin{equation}
\label{906}
f(X^R)-f(X_R) \ge |V_{n,k}(\ul{R})|-1.
\end{equation}
In view of $n(r+\varepsilon) \le \ol{X}_R \le n(\beta-\varepsilon)$, $n(r+\varepsilon) \le \ol{X}_V \le n(r+\varepsilon+\beta+\varepsilon)$ and $r < \beta$ we have
$\ol{X}_R-\ul{X}_V \le n(\beta-r-2\varepsilon) \le b$ and $\ol{X}_V-\ul{X}_R \le n(\beta+\varepsilon) \le b$
if  $\varepsilon$ is sufficiently small and  $n$ is sufficiently large. Thus, by Lemma \ref{164}, $X_V$ and $X_R$ are adjacent and hence
\begin{equation}
\label{913}
f(X^R)-f(X_V) \le B(\G).
\end{equation}
In view of $\ul{X}_R,\ul{X}^V\ge n(r+\varepsilon)$ and analogously to Case 2, $X_R$ and $X^V$ have distance at most $q$ and hence
\begin{equation}
\label{919}
f(X^V)-f(X_R) \le qB(\G).
\end{equation}
From \eqref{906}, \eqref{913} and \eqref{919} we obtain
\[
f(X^V)-f(X_V) \le (q+1)B(\G)-(|V_{n,k}(\ul{R})|-1)
\]
and thus by Corollary \ref{574} and Lemma \ref{628}
\[
B(\G) \gtrsim c_2(\beta,k)n^k+\left(1-\frac{2}{\beta-r} \varepsilon \right)^2 \frac{1}{q^{k-1}}c_3(\beta,k) n^k
\]
and with $\varepsilon \rightarrow 0$ the assertion follows. \qed

\section{Proof of the upper bounds for the bandwidth in Theorem \ref{thm2}}
\label{1122}

In the following, we present two proper numberings $f$ of $\G$ whose bandwidth is asymptotically equal to the asserted upper bounds.
As in Section \ref{sec3}, we define a total order $\V=S_1 \oplus \dots \oplus S_l$ with suborders given by means of polygons. In order to avoid intersections on the boundaries we explicitly describe which part of the boundary is deleted, though the ordering can be given on the whole polygon.
For example, a notation of the form $C_0B_1D_1E_0 \setminus \ol{B_1D_1}$ means that the segment $\ol{B_1D_1}$ is deleted from the closed quadrangle $C_0B_1D_1E_0$.

{\bf Case a)} $r \le \frac{q-1}{q^2+q+1}.$

We define the total order $\le$ as follows:
\begin{multline*}
\V=V_{n,k}(C_0B_1D_1E_0 \setminus \ol{B_1D_1}) \oplus V_{n,k}(B_1C_1E_1D_1 \setminus \ol{C_1E_1}) \oplus V_{n,k}(C_1B_2D_2E_1 \setminus \ol{B_2D_2}) \\
\oplus \dots \oplus V_{n,k}(B_qC_qE_qD_q \setminus \ol{C_qE_q}) \oplus V_{n,k}(C_qB_{q+1}D_{q+1}E_q).
\end{multline*}
We still have to define the ordering of the elements of $V_{n,k}(C_iB_{i+1}D_{i+1}E_i)$, $i=0,\dots,q$, and of $V_{n,k}(B_iC_iE_iD_i)$, $i=1,\dots,q$
(here we may allow the complete boundary).

If $(\ul{X},\ol{X})=(\ul{Y},\ol{Y})$ we set in both cases $X \le_i Y$ if $X \le_{lex} Y$. Thus let $(\ul{X},\ol{X})\neq (\ul{Y},\ol{Y})$.

First we discuss $V_{n,k}(C_iB_{i+1}D_{i+1}E_i)$. We use a new coordinate system with the same origin and with transformation matrix and inverse transformation matrix
\[
M_i=
\begin{pmatrix}
1&-i/q\\1&1-i/q
\end{pmatrix} \text{ and }
M_i^{-1}=
\begin{pmatrix}
1-i/q&i/q\\-1&1
\end{pmatrix}.
\]
Thus the new coordinate axes have direction of $\overrightarrow{C_iB_{i+1}}$ and $\overrightarrow{C_iE_i}$. The ordering is a lexicographic ordering of the points $\frac{1}{n} (\ul{X},\ol{X})$ with respect to the new coordinate system, i.e., for $X,Y \in V_{n,k}(C_iB_{i+1}D_{i+1}E_i)$ and $(\ul{X},\ol{X})\neq (\ul{Y},\ol{Y})$ we set
\[
X \le _i Y \text{ if } ((1-i/q) \ul{X}+i \ol{X}/q,-\ul{X}+\ol{X}) \le_{lex} ((1-i/q) \ul{Y}+i \ol{Y}/q,-\ul{Y}+\ol{Y}),
\]
see Figure \ref{fig:ub1}.

\begin{figure}[H]
\centering
\begin{tikzpicture}[scale=1.5]
   \foreach \s in {5} {
   \foreach \b in {0.45}{
   \foreach \q in {2} {
   \foreach \r in {0.1} {
   \foreach \g in {\b*(1-1/\q)} {

   \coordinate[label=below:{$C_0=(0,0)$}] (0) at (0,0);
   \coordinate[label=below:{$(1,0)$}] (10) at (\s,0);
   \coordinate[label=above:{$(0,1)$}] (01) at (0,\s);
   \coordinate[label=above:{$B_3=(1,1)$}] (11) at (\s,\s);
   \draw[] (0) -- (10) -- (11) -- (01) -- cycle;

   \foreach \i in {1,...,\q} {
   \coordinate[label=above:{$A_{\i}$}] (A\i) at (\s*\i*\r,{\s*(\i*\r+\q*(\b-\r))});
   \coordinate[label=below:{$B_{\i}$}] (B\i) at ({\s*(\r-\b+\i*\b)},{\s*(\r-\b+\i*\b)});
   \coordinate[label=below:{$C_{\i}$}] (C\i) at (\s*\i*\b,\s*\i*\b);
   \coordinate[label=left:{$D_{\i}$}] (D\i) at ({\s*(\r+(\i-1)*\g)},{\s*(\r+(\i-1)*\g+\b)});
   \coordinate[label=above:{$E_{\i}$}] (E\i) at ({\s*(\i*\g)},{\s*(\i*\g+\b)});

   \draw (A\i) -- (B\i);
   \draw (A\i) -- (C\i);
   };
   \coordinate[label=left:{$E_0$}] (E0) at (0,\s*\b);
   \coordinate[label=above:{$D_3$}] (D3) at ({\s*(1-\b)},\s);
   \foreach \i in {A,B,C,D,E} {
   \foreach \j in {1,...,\q} {
   \fill (\i\j) circle(1pt);

   };
   };
   };
   };
   };
   };
   \fill (E0) circle(1pt);
   \fill (D3) circle(1pt);
   \fill (0) circle(1pt);
   \fill (10) circle(1pt);
   \fill (01) circle(1pt);
   \fill (11) circle(1pt);
   \draw (0) -- (11);
   \draw (E0) -- (D3);

   \foreach \x in {0,...,2}
 \draw[blue,
    decoration={markings,mark=at position 1 with {\arrow[scale=2,blue]{>}}},
    postaction={decorate},
    shorten >=0.4pt
    ] ($(0)!\x/3!(B1)$) -- ($(E0)!\x/3!(D1)$);
  \foreach \x in {0,...,6}
 \draw[blue,
    decoration={markings,mark=at position 1 with {\arrow[scale=2,blue]{>}}},
    postaction={decorate},
    shorten >=0.4pt
    ] ($(B1)!\x/7!(C1)$) -- ($(D1)!\x/7!(E1)$);
     \foreach \x in {0,...,2}
 \draw[blue,
    decoration={markings,mark=at position 1 with {\arrow[scale=2,blue]{>}}},
    postaction={decorate},
    shorten >=0.4pt
    ] ($(C1)!\x/3!(B2)$) -- ($(E1)!\x/3!(D2)$);
     \foreach \x in {0,...,6}
 \draw[blue,
    decoration={markings,mark=at position 1 with {\arrow[scale=2,blue]{>}}},
    postaction={decorate},
    shorten >=0.4pt
    ] ($(B2)!\x/7!(C2)$) -- ($(D2)!\x/7!(E2)$);
     \foreach \x in {0,...,3}
 \draw[blue,
    decoration={markings,mark=at position 1 with {\arrow[scale=2,blue]{>}}},
    postaction={decorate},
    shorten >=0.4pt
    ] ($(C2)!\x/3!(11)$) -- ($(E2)!\x/3!(D3)$);
 };

\end{tikzpicture}
\caption{Schematic illustration of the ordering for Case a).}\label{fig:ub1}
\end{figure}
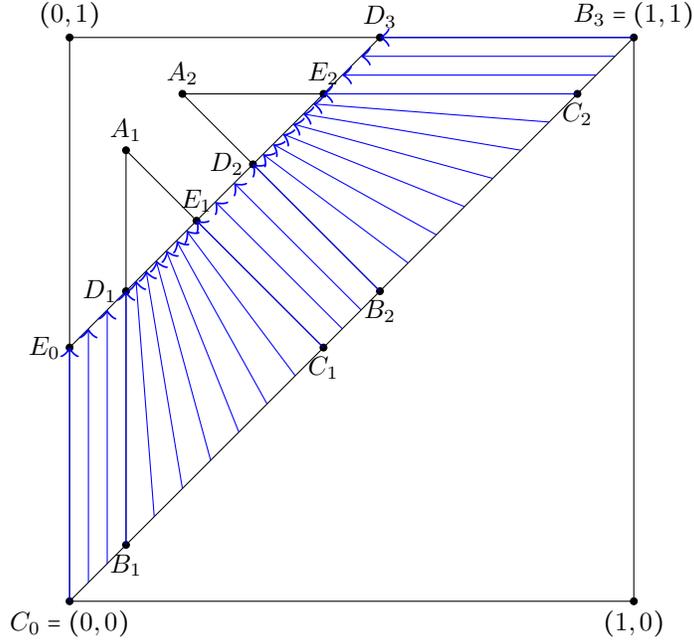

Now we discuss $V_{n,k}(B_iC_iE_iD_i)$. Here we work with polar coordinates in the coordinate system with origin $A_i$ and $x$-axis in the direction of $\overrightarrow{A_iB_i}$ and arbitrary, but fixed unit length. For $X \in V_{n,k}(B_iC_iE_iD_i)$ let $\varphi_i(X)$ and $r_i(X)$ be the angular and radial coordinates of $\frac{1}{n} (\ul{X},\ol{X})$ in this coordinate system. The ordering is a lexicographic ordering with respect to the \emph{reflected} polar coordinates, i.e.,
for $X,Y \in V_{n,k}(B_iC_iE_iD_i)$ and $(\ul{X},\ol{X})\neq (\ul{Y},\ol{Y})$ we set
\[
X \le_i Y \text{ if } (\varphi_i(X),-r_i(X)) \le_{lex} (\varphi_i(Y),-r_i(Y)),
\]
see Figure \ref{fig:ub1}.

Note that for simpler numerical computations $\varphi_i(X)$ may be enlarged to an angle such that one leg is parallel to the $y$-axis, the size of the angle may be replaced by $\tan(\varphi_i(X))$ and the Euclidean norm for $r_i(X)$ may be replaced by some other norm, e.g. the $L_1$-norm.

It is easy to check that
\begin{equation}
\label{1016}
X \le Y \text{ implies } \ul{X} \le \ol{Y}.
\end{equation}

\begin{lem}
\label{1021}
Let $f$ be the numbering for Case a). Then, for $n \rightarrow \infty$,
\[
B_f(G) \lesssim c_1(\beta,k)n^k.
\]
\end{lem}
\proof
Let $P=(\xi,\xi)$ be any point on the segment $\overline{C_0B_{q+1}}$, i.e., $0 \le \xi \le 1$.
With $P$ we associate a new point $\hat{P}$ as follows:
If $P \in \overline{C_iB_{i+1}}$ for some $i$ then let $\hat{P}$ be the intersection point of the line $g_{y=x+\beta}$ with the line through $P$ that is parallel to $C_iE_i$.
If $P \in \overline{B_iC_{i}}$ for some $i$ then let $\hat{P}$ be the intersection point of the line $g_{y=x+\beta}$ with the line through $P$ and $A_i$, see Figure \ref{fig:ub2}.

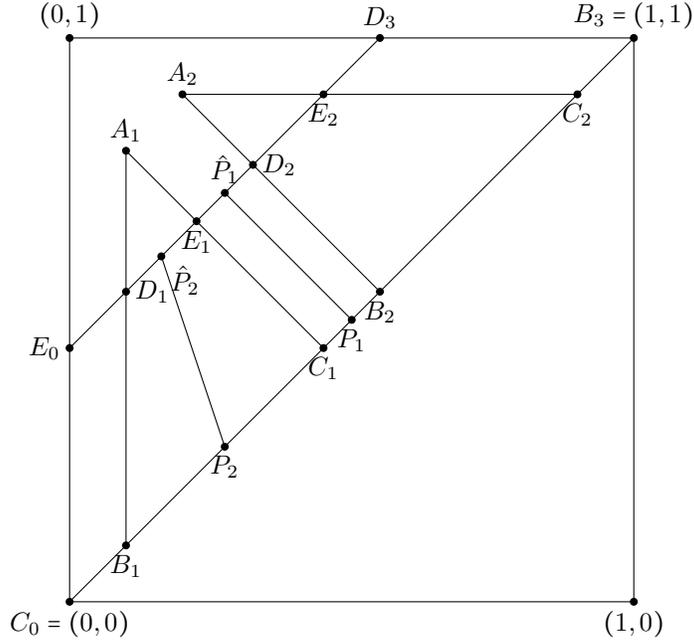
\begin{figure}[H]
\centering
\begin{tikzpicture}[scale=1.5]
   \foreach \s in {5} {
   \foreach \b in {0.45}{
   \foreach \q in {2} {
   \foreach \r in {0.1} {
   \foreach \g in {\b*(1-1/\q)} {

   \coordinate[label=below:{$C_0=(0,0)$}] (0) at (0,0);
   \coordinate[label=below:{$(1,0)$}] (10) at (\s,0);
   \coordinate[label=above:{$(0,1)$}] (01) at (0,\s);
   \coordinate[label=above:{$B_3=(1,1)$}] (11) at (\s,\s);
   \draw[] (0) -- (10) -- (11) -- (01) -- cycle;

   \foreach \i in {1,...,\q} {
   \coordinate[label=above:{$A_{\i}$}] (A\i) at (\s*\i*\r,{\s*(\i*\r+\q*(\b-\r))});
   \coordinate[label=below:{$B_{\i}$}] (B\i) at ({\s*(\r-\b+\i*\b)},{\s*(\r-\b+\i*\b)});
   \coordinate[label=below:{$C_{\i}$}] (C\i) at (\s*\i*\b,\s*\i*\b);
   \coordinate[label=right:{$D_{\i}$}] (D\i) at ({\s*(\r+(\i-1)*\g)},{\s*(\r+(\i-1)*\g+\b)});
   \coordinate[label=below:{$E_{\i}$}] (E\i) at ({\s*(\i*\g)},{\s*(\i*\g+\b)});

   \draw (A\i) -- (B\i);
   \draw (A\i) -- (C\i);
   };
   \coordinate[label=left:{$E_0$}] (E0) at (0,\s*\b);
   \coordinate[label=above:{$D_3$}] (D3) at ({\s*(1-\b)},\s);
   \foreach \i in {A,B,C,D,E} {
   \foreach \j in {1,...,\q} {
   \fill (\i\j) circle(1pt);

   };
   };
   };
   };
   };
   };
   \coordinate[label=below:$P_2$] (P2) at ($(B1)!0.5!(C1)$);
   \coordinate[label=below:$P_1$] (P1) at ($(C1)!0.5!(B2)$);
   \coordinate[label=-20:$\hat{P}_2$] (Ph2) at
(intersection of A1--P2 and D1--E1);
\coordinate[label=above:$\hat{P}_1$] (Ph1) at
(intersection of 01--P1 and E1--D2);
   \fill (P1) circle(1pt);
   \fill (P2) circle(1pt);
   \fill (Ph1) circle(1pt);
   \fill (Ph2) circle(1pt);
   \fill (E0) circle(1pt);
   \fill (D3) circle(1pt);
   \fill (0) circle(1pt);
   \fill (10) circle(1pt);
   \fill (01) circle(1pt);
   \fill (11) circle(1pt);
   \draw (0) -- (11);
   \draw (E0) -- (D3);
   \draw (Ph1) -- (P1);
   \draw (Ph2) -- (P2);

 };

\end{tikzpicture}
\caption{Important points for the upper bound for Case a).}\label{fig:ub2}
\end{figure}

Moreover, for $X \in  \V$ let
\begin{equation}
\label{1281}
P_{\ul{X}}=\frac{1}{n}(\ul{X},\ul{X}) \text{ and } P_{\ol{X}}=\frac{1}{n}(\ol{X},\ol{X}).
\end{equation}
Let $X,Y \in \V$ with $X \le Y$.
By the definition of the ordering and in view of \eqref{1016},
\[
f(Y)-f(X) \le |V_{n,k}(P_{\ul{X}}P_{\ol{Y}}\hat{P}_{\ol{Y}}\hat{P}_{\ul{X}})|.
\]
Thus we have to prove that
\begin{equation}
\label{1042}
|V_{n,k}(P_{\ul{X}}P_{\ol{Y}}\hat{P}_{\ol{Y}}\hat{P}_{\ul{X}})| \lesssim c_1(\beta,k)n^k.
\end{equation}

Note that
\begin{equation}
\label{1038}
\overrightarrow{B_iB_{i+1}}=\overrightarrow{C_iC_{i+1}}=(\beta,\beta).
\end{equation}
Since $X$ and $Y$ are adjacent we have by Lemma \ref{164}, $\ol{Y}-\ul{X} \le b \sim \beta n$.
Let $\varepsilon > 0$. Then $\frac{1}{n}(\ol{Y}-\ul{X}) \le \beta+\varepsilon$ for sufficiently large $n$.
Let
\begin{equation}
\label{1305}
P'=(\frac{1}{n}\ul{X}+\beta,\frac{1}{n}\ul{X}+\beta).
\end{equation}
If $\beta < \frac{1}{n}(\ol{Y}-\ul{X}) \le \beta+\varepsilon$ then $\overrightarrow{P'P_{\ol{Y}}}=(\delta,\delta)$ with $\delta\leq \varepsilon$.
In view of
\[
\mu(P_{\ul{X}}P_{\ol{Y}}\hat{P}_{\ol{Y}}\hat{P}_{\ul{X}})=\mu(P_{\ul{X}}P'\hat{P}'\hat{P}_{\ul{X}})+\mu(P'P_{\ol{Y}}\hat{P}_{\ol{Y}}\hat{P}')
\]
we have
\[
\mu(P_{\ul{X}}P_{\ol{Y}}\hat{P}_{\ol{Y}}\hat{P}_{\ul{X}})\le \mu(P_{\ul{X}}P'\hat{P}'\hat{P}_{\ul{X}})+O(\varepsilon),
\]
which is clearly also true if $\frac{1}{n}(\ol{Y}-\ul{X}) \le \beta$.
Using Corollary \ref{574}, it is easy to check that for all possible $i$,
\begin{equation}\label{1222}
\begin{aligned}
\mu(P_{\ul{X}}P'\hat{P}'\hat{P}_{\ul{X}})&=\mu(B_iB_{i+1}D_{i+1}D_i) = \mu(C_iC_{i+1}E_{i+1}E_i)\\
&=\mu(B_iC_iE_iD_i) + \mu(C_iB_{i+1}D_{i+1}E_i)\\
&=c_1(\beta,k).
\end{aligned}
\end{equation}

An illustration of this fact can be found in Figures \ref{fig:ub3} and \ref{fig:ub4b}.

\begin{figure}[H]
\centering
\begin{tikzpicture}[scale=1.5]
   \foreach \s in {5} {
   \foreach \b in {0.45}{
   \foreach \q in {2} {
   \foreach \r in {0.1} {
   \foreach \g in {\b*(1-1/\q)} {

   \coordinate[label=below:{$C_0=(0,0)$}] (0) at (0,0);
   \coordinate[label=below:{$(1,0)$}] (10) at (\s,0);
   \coordinate[label=above:{$(0,1)$}] (01) at (0,\s);
   \coordinate[label=above:{$B_3=(1,1)$}] (11) at (\s,\s);
   \draw[] (0) -- (10) -- (11) -- (01) -- cycle;

   \foreach \i in {1,...,\q} {
   \coordinate[label=above:{$A_{\i}$}] (A\i) at (\s*\i*\r,{\s*(\i*\r+\q*(\b-\r))});
   \coordinate[label=right:{$B_{\i}$}] (B\i) at ({\s*(\r-\b+\i*\b)},{\s*(\r-\b+\i*\b)});
   \coordinate[label=below:{$C_{\i}$}] (C\i) at (\s*\i*\b,\s*\i*\b);
   \coordinate[label=right:{$D_{\i}$}] (D\i) at ({\s*(\r+(\i-1)*\g)},{\s*(\r+(\i-1)*\g+\b)});
   \coordinate[label=below:{$E_{\i}$}] (E\i) at ({\s*(\i*\g)},{\s*(\i*\g+\b)});

   \draw (A\i) -- (B\i);
   \draw (A\i) -- (C\i);
   };
   \coordinate[label=left:{$E_0$}] (E0) at (0,\s*\b);
   \coordinate[label=above:{$D_3$}] (D3) at ({\s*(1-\b)},\s);
   \foreach \i in {A,B,C,D,E} {
   \foreach \j in {1,...,\q} {
   \fill (\i\j) circle(1pt);

   };
   };
   \coordinate[label=right:$P_{\ul{X}}$] (P1) at ($(0)!0.5!(B1)$);
   \coordinate[label=right:$P^{'}$] (P2) at ($(P1)+({\s*\b},{\s*\b})$);
   \coordinate[label=above:$\hat{P}_{\ul{X}}$] (Ph1) at ($(P1)+(0,{\s*\b})$);
\coordinate[label=above:$\hat{P}^{'}$] (Ph2) at
(intersection of 01--P2 and E1--D2);
   };
   };
   };
   };

   \fill (P1) circle(1pt);
   \fill (P2) circle(1pt);
   \fill (Ph1) circle(1pt);
   \fill (Ph2) circle(1pt);
   \fill (E0) circle(1pt);
   \fill (D3) circle(1pt);
   \fill (0) circle(1pt);
   \fill (10) circle(1pt);
   \fill (01) circle(1pt);
   \fill (11) circle(1pt);
   \draw (0) -- (11);
   \draw (E0) -- (D3);
   \draw (Ph2) -- (P2);
   \draw (Ph1) -- (P1);

   \filldraw[fill=red!20!white] (P1) -- (B1) -- (D1) -- (Ph1) --  cycle;
   \filldraw[fill=red!20!white] (P2) -- (B2) -- (D2) -- (Ph2) --  cycle;
 };

\end{tikzpicture}
\caption{Illustration of \eqref{1222}, where $P_{\ul{X}} \in \overline{C_iB_{i+1}}$ for some $i$. Both red quadrangles have the same measure by Corollary \ref{883}.}
\label{fig:ub3}
\end{figure}

 \begin{figure}[H]
 \centering
\begin{tikzpicture}[scale=1.5]
   \foreach \s in {5} {
   \foreach \b in {0.45}{
   \foreach \q in {2} {
   \foreach \r in {0.1} {
   \foreach \g in {\b*(1-1/\q)} {

   \coordinate[label=below:{$C_0=(0,0)$}] (0) at (0,0);
   \coordinate[label=below:{$(1,0)$}] (10) at (\s,0);
   \coordinate[label=above:{$(0,1)$}] (01) at (0,\s);
   \coordinate[label=above:{$B_3=(1,1)$}] (11) at (\s,\s);
   \draw[] (0) -- (10) -- (11) -- (01) -- cycle;

   \foreach \i in {1,...,\q} {
   \coordinate[label=above:{$A_{\i}$}] (A\i) at (\s*\i*\r,{\s*(\i*\r+\q*(\b-\r))});
   \coordinate[label=right:{$B_{\i}$}] (B\i) at ({\s*(\r-\b+\i*\b)},{\s*(\r-\b+\i*\b)});
   \coordinate[label=below:{$C_{\i}$}] (C\i) at (\s*\i*\b,\s*\i*\b);
   \coordinate[label=left:{$D_{\i}$}] (D\i) at ({\s*(\r+(\i-1)*\g)},{\s*(\r+(\i-1)*\g+\b)});
   \coordinate[label=above:{$E_{\i}$}] (E\i) at ({\s*(\i*\g)},{\s*(\i*\g+\b)});

   \draw (A\i) -- (B\i);
   \draw (A\i) -- (C\i);
   };
   \coordinate[label=left:{$E_0$}] (E0) at (0,\s*\b);
   \coordinate[label=above:{$D_3$}] (D3) at ({\s*(1-\b)},\s);
   \foreach \i in {A,B,C,D,E} {
   \foreach \j in {1,...,\q} {
   \fill (\i\j) circle(1pt);

   };
   };
   \coordinate[label=right:$P_{\ul{X}}$] (P1) at ($(B1)!0.25!(C1)$);
   \coordinate[label=right:$P^{'}$] (P2) at ($(P1)+({\s*\b},{\s*\b})$);
   \coordinate[label=above:$\hat{P}_{\ul{X}}$] (Ph1) at (intersection of A1--P1 and E0--D3);
\coordinate[label=above:$\hat{P}^{'}$] (Ph2) at
(intersection of A2--P2 and E0--D3);
   };
   };
   };
   };

   \fill (P1) circle(1pt);
   \fill (P2) circle(1pt);
   \fill (Ph1) circle(1pt);
   \fill (Ph2) circle(1pt);
   \fill (E0) circle(1pt);
   \fill (D3) circle(1pt);
   \fill (0) circle(1pt);
   \fill (10) circle(1pt);
   \fill (01) circle(1pt);
   \fill (11) circle(1pt);
   \draw (0) -- (11);
   \draw (E0) -- (D3);
   \draw (Ph2) -- (P2);
   \draw (Ph1) -- (P1);

   \filldraw[fill=red!20!white] (P1) -- (B1) -- (D1) -- (Ph1) --  cycle;
   \filldraw[fill=red!20!white] (P2) -- (B2) -- (D2) -- (Ph2) --  cycle;

 };
 \end{tikzpicture}
 \caption{Illustration of \eqref{1222},  where $P_{\ul{X}} \in \overline{B_iC_i}$ for some $i$. Both red quadrangles have the same measure by Corollary \ref{883} and the intercept theorem.}
\label{fig:ub4b}
 \end{figure}

By Corollary \ref{cor2},
\[
|V_{n,k}(P_{\ul{X}}P_{\ol{Y}}\hat{P}_{\ol{Y}}\hat{P}_{\ul{X}})| \lesssim \left(c_1(\beta,k)+O(\varepsilon)\right)n^k
\]
and with $\varepsilon \rightarrow 0$ we get \eqref{1042}.

{\bf Case b)} $r > \frac{q-1}{q^2+q+1}.$

For this case, we use an ordering similar to Case a), but with different polygons, due to the different location of their defining points.
Let $A_0=(0,q(\beta -r))$, $A_{q+1}=((q+1)r,1))$ and $I=(0,1)$.
Note that the points $A_i, i=0,\dots,q+1$, lie on the line $g_{y=x+q(\beta-r)}$.
We define the total order $\le $ as follows:
\begin{align*}
\V=&V_{n,k}(C_0B_1A_1G_1G_0A_0 \setminus (\ol{A_1B_1} \cup \ol{A_1G_1})) \oplus V_{n,k}(A_1B_1C_1 \setminus \ol{A_1C_1})\\
 &\oplus V_{n,k}(C_1B_2A_2G_2G_1A_1 \setminus (\ol{A_2B_2} \cup \ol{A_2G_2})) \oplus \dots \\
 &\oplus V_{n,k}(C_{q-1}B_qA_qG_qG_{q-1}A_{q-1} \setminus (\ol{A_qB_q}\cup \ol{A_qG_q})) \\
&\oplus V_{n,k}(A_qB_qC_q \setminus \ol{A_qC_q}) \oplus V_{n,k}(C_qB_{q+1}G_{q+1}G_qA_q).
\end{align*}
We still have to define the ordering of the elements of $V_{n,k}(C_iB_{i+1}A_{i+1}G_{i+1}G_iA_i)$, $i=0,\dots,q$, and of $V_{n,k}(A_iB_iC_i)$, $i=1,\dots,q$ (again, we may allow the complete boundary).

If $(\ul{X},\ol{X})=(\ul{Y},\ol{Y})$ we set in all cases $X \le_i Y$ if $X \le_{lex} Y$. Thus let $(\ul{X},\ol{X})\neq (\ul{Y},\ol{Y})$.

First we discuss $V_{n,k}(C_iB_{i+1}A_{i+1}G_{i+1}G_iA_i)$. We divide the hexagon $C_iB_{i+1}A_{i+1}G_{i+1}G_iA_i$ into two quadrangles $C_iB_{i+1}A_{i+1}A_i$ and $A_iA_{i+1}G_{i+1}G_i$, define the corresponding orderings for both quadrangles and then explain how they are combined.

The definition of the ordering of $V_{n,k}(C_iB_{i+1}A_{i+1}A_i)$ is similar to $V_{n,k}(C_iB_{i+1}D_{i+1}E_i)$ of Case a), i.e., for $X,Y \in V_{n,k}(C_iB_{i+1}A_{i+1}A_i)$ and $(\ul{X},\ol{X}) \neq (\ul{Y},\ol{Y})$ we set
\[
X \le _i Y \text{ if } ((1-i/q) \ul{X}+i \ol{X}/q,-\ul{X}+\ol{X}) \le_{lex} ((1-i/q) \ul{Y}+i \ol{Y}/q,-\ul{Y}+\ol{Y}).
\]

Concerning $V_{n,k}(A_iA_{i+1}G_{i+1}G_i)$ we work with polar coordinates $\varphi_i(X)$ and $r_i(X)$ of points $\frac{1}{n}(\ul{X},\ol{X})$, where $X \in V_{n,k}(A_iA_{i+1}G_{i+1}G_i)$, in the coordinate system with origin $I$ and $x$-axis in the direction of $\overrightarrow{I A_i}$ and arbitrary, but fixed unit length. Similarly to $V_{n,k}(B_iC_iE_iD_i)$ in Case a), the ordering is a lexicographic ordering with respect to the reflected polar coordinates, i.e., for $X,Y \in V_{n,k}(A_iA_{i+1}G_{i+1}G_i)$ and $(\ul{X},\ol{X}) \neq (\ul{Y},\ol{Y})$ we set
\[
X \le_i Y \text{ if } (\varphi_i(X),-r_i(X)) \le_{lex} (\varphi_i(Y),-r_i(Y)).
\]
For a point $P\in C_iB_{i+1}A_{i+1}A_i$ let $\tilde{P}$ be the intersection point of the line $g_{y=x+q(\beta-r)}$ with the line through $P$ that is parallel to $C_iA_i$. If, in particular, $P=\frac{1}{n}(\ul{X},\ol{X})$ with $X \in V_{n,k}(C_iB_{i+1}A_{i+1}A_i)$, then let $\tilde{\varphi}(X)$ be the angular coordinate of $\tilde{P}$ in the coordinate system introduced for $V_{n,k}(A_iA_{i+1}G_{i+1}G_i)$.

The combination of the two orderings is as follows: Let $X \in V_{n,k}(C_iB_{i+1}A_{i+1}A_i)$ and $Y \in V_{n,k}(A_iA_{i+1}G_{i+1}G_i)$ and $(\ul{X},\ol{X}) \neq (\ul{Y},\ol{Y})$. We set
\[
X \le _i Y \text{ if } \tilde{\varphi}_i(X) \le \varphi_i(Y).
\]

Finally, we discuss $V_{n,k}(A_iB_iC_i)$. Here the ordering is a lexicographic ordering of the reflected polar coordinates in the coordinate system with origin $A_i$ and $x$-axis in the direction of $\overrightarrow{A_iB_i}$  and arbitrary, but fixed unit length.

The whole ordering is illustrated in Figure \ref{fig:ub5}.

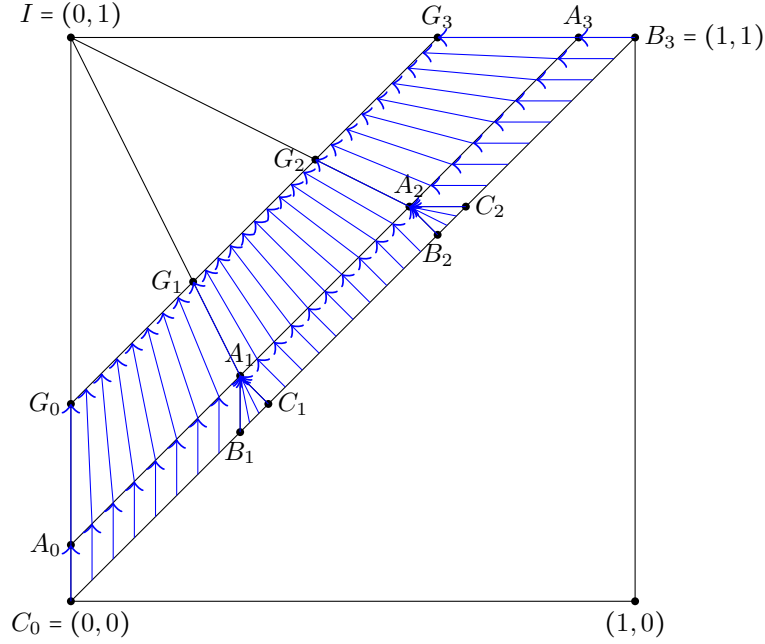
\begin{figure}[H]
\centering
 \begin{tikzpicture}[scale=1.5]
   \foreach \s in {5} {
   \foreach \b in {0.35}{
   \foreach \q in {2} {
   \foreach \r in {0.3} {
   \foreach \g in {\b*(1-1/\q)} {

   \coordinate[label=below:{$C_0=(0,0)$}] (0) at (0,0);
   \coordinate[label=below:{$(1,0)$}] (10) at (\s,0);
   \coordinate[label=above:{$I=(0,1)$}] (01) at (0,\s);
   \coordinate[label=right:{$B_3=(1,1)$}] (11) at (\s,\s);
   \draw[] (0) -- (10) -- (11) -- (01) -- cycle;

   \foreach \i in {1,...,\q} {
   \coordinate[label=above:{$A_{\i}$}] (A\i) at ({\s*\i*\r},{\s*(\i*\r+\q*(\b-\r))});
   \coordinate[label=below:{$B_{\i}$}] (B\i) at ({\s*(\r-\b+\i*\b)},{\s*(\r-\b+\i*\b)});
   \coordinate[label=right:{$C_{\i}$}] (C\i) at ({\s*\i*\b},\s*\i*\b);
  \coordinate[label=left:{$G_{\i}$}] (G\i) at ({\s*(\i*(1-\b)/(\q+1))},{\s*(\i*(1-\b)/(\q+1)+\b)});

   \draw (A\i) -- (B\i);
   \draw (A\i) -- (C\i);
   \draw (01) -- (A\i);
   };
   \coordinate[label=left:{$G_{0}$}] (G0) at (0,\s*\b);
   \coordinate[label=above:{$G_{3}$}] (G3) at ({\s*(1-\b)},\s);
   \coordinate[label=left:{$A_{0}$}] (A0) at (0,{\s*(\q*(\b-\r))});
   \coordinate[label=above:{$A_{3}$}] (A3) at ({\s*(\q+1)*\r},{\s*((\q+1)*\r+\q*(\b-\r))});
   \foreach \i in {A,B,C,G} {
   \foreach \j in {1,...,\q} {
   \fill (\i\j) circle(1pt);

   };
   };
   };
   };
   };
   };
   \fill (G0) circle(1pt);
   \fill (A0) circle(1pt);
   \fill (A3) circle(1pt);
   \fill (G3) circle(1pt);
   \fill (0) circle(1pt);
   \fill (10) circle(1pt);
   \fill (01) circle(1pt);
   \fill (11) circle(1pt);
   \draw (0) -- (11);
   \draw (A0) -- (A3);
   \draw (G0) -- (G3);

   \foreach \x in {0,...,7}
 \draw[blue,
    decoration={markings,mark=at position 1 with {\arrow[scale=2,blue]{>}}},
    postaction={decorate},
    shorten >=0.4pt
    ] ($(0)!\x/8!(B1)$) -- ($(A0)!\x/8!(A1)$);

    \foreach \x in {0,...,7}
 \draw[blue,
    decoration={markings,mark=at position 1 with {\arrow[scale=2,blue]{>}}},
    postaction={decorate},
    shorten >=0.4pt
    ] ($(A0)!\x/8!(A1)$) -- ($(G0)!\x/8!(G1)$);

    \foreach \x in {0,...,2}
   \draw[blue,
    decoration={markings,mark=at position 1 with {\arrow[scale=2,blue]{>}}},
    postaction={decorate},
    shorten >=0.4pt
    ] ($(B1)!\x/3!(C1)$) -- (A1);

    \foreach \x in {0,...,9}
   \draw[blue,
    decoration={markings,mark=at position 1 with {\arrow[scale=2,blue]{>}}},
    postaction={decorate},
    shorten >=0.4pt
    ] ($(C1)!\x/10!(B2)$) -- ($(A1)!\x/10!(A2)$);

 \foreach \x in {0,...,9}
   \draw[blue,
    decoration={markings,mark=at position 1 with {\arrow[scale=2,blue]{>}}},
    postaction={decorate},
    shorten >=0.4pt
    ] ($(A1)!\x/10!(A2)$) -- ($(G1)!\x/10!(G2)$);

    \foreach \x in {0,...,2}
   \draw[blue,
    decoration={markings,mark=at position 1 with {\arrow[scale=2,blue]{>}}},
    postaction={decorate},
    shorten >=0.4pt
    ] ($(B2)!\x/3!(C2)$) -- (A2);

     \foreach \x in {0,...,8}
   \draw[blue,
    decoration={markings,mark=at position 1 with {\arrow[scale=2,blue]{>}}},
    postaction={decorate},
    shorten >=0.4pt
    ] ($(C2)!\x/8!(11)$) -- ($(A2)!\x/8!(A3)$);

    \foreach \x in {0,...,8}
   \draw[blue,
    decoration={markings,mark=at position 1 with {\arrow[scale=2,blue]{>}}},
    postaction={decorate},
    shorten >=0.4pt
    ] ($(A2)!\x/8!(A3)$) -- ($(G2)!\x/8!(G3)$);

 };
\end{tikzpicture}
\caption{Schematic illustration of the ordering for Case b).}\label{fig:ub5}
\end{figure}

\begin{lem}
\label{1454}
Let $f$ be the numbering for Case b). Then
\[
B_f(G) \lesssim (c_2(\beta,k)+c_3(\beta,k))n^k.
\]
\end{lem}
\proof
Let $P=(\xi,\xi)$ be any point on the segment $\overline{C_0B_{q+1}}$, i.e., $0 \le \xi \le 1$.
With $P$ we associate a new point $\hat{P}$ as follows:
If $P \in \overline{C_iB_{i+1}}$ for some $i$, i.e., $P$ belongs to the quadrangle $C_iB_{i+1}A_{i+1}A_i$, then we already defined $\tilde{P}$. The point $\hat{P}$ is the intersection point of the line $g_{y=x+\beta}$ with the line through $\tilde{P}$ and $I$.
If $P \in \overline{B_iC_{i}}$ for some $i$ then let $\hat{P}$ be the intersection point of the line $g_{y=x+\beta}$ with the line through $I$ and $A_i$, see Figure \ref{fig:ub6}.

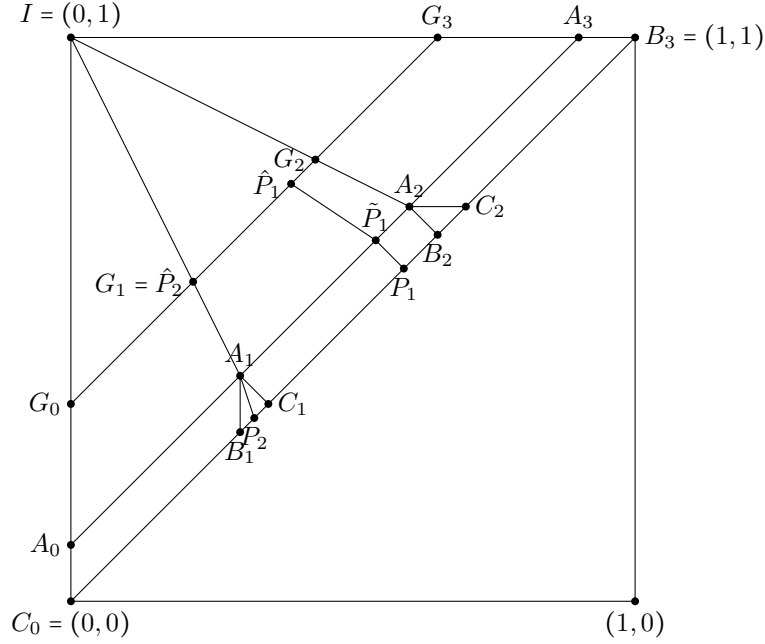
\begin{figure}[H]
\centering
 \begin{tikzpicture}[scale=1.5]
   \foreach \s in {5} {
   \foreach \b in {0.35}{
   \foreach \q in {2} {
   \foreach \r in {0.3} {
   \foreach \g in {\b*(1-1/\q)} {

   \coordinate[label=below:{$C_0=(0,0)$}] (0) at (0,0);
   \coordinate[label=below:{$(1,0)$}] (10) at (\s,0);
   \coordinate[label=above:{$I=(0,1)$}] (01) at (0,\s);
   \coordinate[label=right:{$B_3=(1,1)$}] (11) at (\s,\s);
   \coordinate[label=left:{$G_{1}=\hat{P}_2$}] (G1) at ({\s*(1*(1-\b)/(\q+1))},{\s*(1*(1-\b)/(\q+1)+\b)});
   \coordinate[label=left:{$G_{2}$}] (G2) at ({\s*(2*(1-\b)/(\q+1))},{\s*(2*(1-\b)/(\q+1)+\b)});
   \draw[] (0) -- (10) -- (11) -- (01) -- cycle;

   \foreach \i in {1,...,\q} {
   \coordinate[label=above:{$A_{\i}$}] (A\i) at ({\s*\i*\r},{\s*(\i*\r+\q*(\b-\r))});
   \coordinate[label=below:{$B_{\i}$}] (B\i) at ({\s*(\r-\b+\i*\b)},{\s*(\r-\b+\i*\b)});
   \coordinate[label=right:{$C_{\i}$}] (C\i) at ({\s*\i*\b},\s*\i*\b);

   \draw (A\i) -- (B\i);
   \draw (A\i) -- (C\i);
   \draw (01) -- (A\i);
   };
   \coordinate[label=left:{$G_{0}$}] (G0) at (0,\s*\b);
   \coordinate[label=above:{$G_{3}$}] (G3) at ({\s*(1-\b)},\s);
   \coordinate[label=left:{$A_{0}$}] (A0) at (0,{\s*(\q*(\b-\r))});
   \coordinate[label=above:{$A_{3}$}] (A3) at ({\s*(\q+1)*\r},{\s*((\q+1)*\r+\q*(\b-\r))});
   \foreach \i in {A,B,C,G} {
   \foreach \j in {1,...,\q} {
   \fill (\i\j) circle(1pt);

   };
   };

   };
   };
   };
   };
   \coordinate[label=below:$P_2$] (P2) at ($(B1)!0.5!(C1)$);
   \coordinate[label=below:$P_1$] (P1) at ($(C1)!0.8!(B2)$);
   \coordinate (P1h) at ($(P1)+(A2)-(B2)$);
   \coordinate[label=above:$\tilde{P}_1$] (Pt1) at
(intersection of P1--P1h and A1--A2);
\coordinate[label=left:$\hat{P}_1$] (Ph1) at
(intersection of 01--Pt1 and G1--G2);
   \fill (P1) circle(1pt);
   \fill (P2) circle(1pt);
   \fill (Pt1) circle(1pt);
   \fill (Ph1) circle(1pt);

   \fill (G0) circle(1pt);
   \fill (A0) circle(1pt);
   \fill (A3) circle(1pt);
   \fill (G3) circle(1pt);
   \fill (0) circle(1pt);
   \fill (10) circle(1pt);
   \fill (01) circle(1pt);
   \fill (11) circle(1pt);
   \draw (0) -- (11);
   \draw (A0) -- (A3);
   \draw (G0) -- (G3);
   \draw (P1) -- (Pt1);
   \draw (Pt1) -- (Ph1);
   \draw (P2) -- (A1);
 };
\end{tikzpicture}
\caption{Important points for the upper bound for Case b).}\label{fig:ub6}
\end{figure}

For $X \in \V$ we define $P_{\ul{X}}$ and $P_{\ol{X}}$ as in \eqref{1281} and, analogously to Case a), have to prove that for $n \rightarrow \infty$
\begin{equation}
\label{1043}
|V_{n,k}(P_{\ul{X}}P_{\ol{Y}}\hat{P}_{\ol{Y}}\hat{P}_{\ul{X}})| \lesssim (c_2(\beta,k)+c_3(\beta,k))n^k.
\end{equation}
We define $P'$ as in \eqref{1305}. With the same arguments as for Case a) it is sufficient to prove that
\[
\mu(P_{\ul{X}}P'\hat{P'}\hat{P}_{\ul{X}})=c_2(\beta,k)+c_3(\beta,k).
\]
Using Lemma \ref{574}, one can verify that for all possible $i$,
\begin{equation}\label{1541}
\begin{aligned}
\mu(P_{\ul{X}}P'\hat{P'}\hat{P}_{\ul{X}})&=\mu(B_iB_{i+1}A_{i+1}G_{i+1}G_iA_i) = \mu(C_iC_{i+1}A_{i+1}G_{i+1}G_iA_i)\\
&=\mu(F_iF_{i+1}G_{i+1}G_i) + \mu(A_{i+1}F_{i+1}C_{i+1})-\mu(A_{i}F_{i}C_{i})\\
&=c_2(\beta,k)+c_3(\beta,k).
\end{aligned}
\end{equation}
An illustration of this fact can be found in Figures \ref{fig:ub7a} and \ref{fig:ub7b}.

\begin{figure}[H]
\centering
\begin{tikzpicture}[scale=1.5]
   \foreach \s in {5} {
   \foreach \b in {0.35}{
   \foreach \q in {2} {
   \foreach \r in {0.3} {
   \foreach \g in {\b*(1-1/\q)} {

   \coordinate[label=below:{$C_0=F_0=(0,0)$}] (0) at (0,0);
   \coordinate[label=below:{$(1,0)$}] (10) at (\s,0);
   \coordinate[label=above:{$I=(0,1)$}] (01) at (0,\s);
   \coordinate[label=right:{$B_3=F_3=(1,1)$}] (11) at (\s,\s);
   \coordinate[label=left:{$G_{1}=\hat{P}_2$}] (G1) at ({\s*(1*(1-\b)/(\q+1))},{\s*(1*(1-\b)/(\q+1)+\b)});
   \coordinate[label=left:{$G_{2}$}] (G2) at ({\s*(2*(1-\b)/(\q+1))},{\s*(2*(1-\b)/(\q+1)+\b)});
   \draw[] (0) -- (10) -- (11) -- (01) -- cycle;

   \foreach \i in {1,...,\q} {
   \coordinate[label=left:{$A_{\i}$}] (A\i) at ({\s*\i*\r},{\s*(\i*\r+\q*(\b-\r))});
   \coordinate[label=below:{$B_{\i}$}] (B\i) at ({\s*(\r-\b+\i*\b)},{\s*(\r-\b+\i*\b)});
   \coordinate[label=right:{$C_{\i}$}] (C\i) at ({\s*\i*\b},\s*\i*\b);
   \coordinate[label=-45:{$F_{\i}$}] (F\i) at ({\s*(\i/(\q+1))},{\s*(\i/(\q+1))});

   \draw (A\i) -- (B\i);
   \draw (A\i) -- (C\i);
   \draw (01) -- (F\i);
   };
   \coordinate[label=left:{$G_{0}$}] (G0) at (0,\s*\b);
   \coordinate[label=above:{$G_{3}$}] (G3) at ({\s*(1-\b)},\s);
   \coordinate[label=left:{$A_{0}$}] (A0) at (0,{\s*(\q*(\b-\r))});
   \coordinate[label=above:{$A_{3}$}] (A3) at ({\s*(\q+1)*\r},{\s*((\q+1)*\r+\q*(\b-\r))});

 \coordinate[label=right:$P_{\ul{X}}$] (P1) at ($(0)!0.2!(B1)$);
   \coordinate[label=right:$P^{'}$] (P2) at ($(P1)+({\s*\b},{\s*\b})$);
   \coordinate (P1h) at ($(P1)+(A1)-(B1)$);
   \coordinate (Pt1) at
(intersection of A0--A3 and P1--P1h);
\coordinate[label=above:$\hat{P}_{\ul{X}}$] (Ph1) at
(intersection of 01--Pt1 and G0--G1);
\coordinate (P2h) at ($(P2)+(A2)-(B2)$);
   \coordinate (Pt2) at
(intersection of A0--A3 and P2--P2h);
\coordinate[label=above:$\hat{P}^{'}$] (Ph2) at
(intersection of 01--Pt2 and G1--G2);
   };
   \filldraw[fill=red!20!white] (0) -- (A0) -- (Pt1) -- (P1) --  cycle;
   \filldraw[fill=red!20!white] (C1) -- (P2) -- (Pt2) -- (A1) --  cycle;
   \filldraw[fill=blue!20!white] (A0) -- (G0) -- (Ph1) -- (Pt1) -- cycle;
   \filldraw[fill=blue!20!white] (A1) -- (G1) -- (Ph2) -- (Pt2) -- cycle;

   \foreach \i in {A,B,C,F,G} {
   \foreach \j in {1,...,\q} {
   \fill (\i\j) circle(1pt);

   };
   };
   };
   };
   };

   \fill (P1) circle(1pt);
   \fill (P2) circle(1pt);
   \fill (Pt1) circle(1pt);
   \fill (Ph1) circle(1pt);
   \fill (Pt2) circle(1pt);
   \fill (Ph2) circle(1pt);

   \fill (G0) circle(1pt);
   \fill (A0) circle(1pt);
   \fill (A3) circle(1pt);
   \fill (G3) circle(1pt);
   \fill (0) circle(1pt);
   \fill (10) circle(1pt);
   \fill (01) circle(1pt);
   \fill (11) circle(1pt);
   \draw (0) -- (11);
   \draw (A0) -- (A3);
   \draw (G0) -- (G3);
   \draw (P1) -- (Pt1);
   \draw (Pt1) -- (Ph1);
   \draw (P2) -- (Pt2);
   \draw (Pt2) -- (Ph2);

 };
\end{tikzpicture}
\caption{Illustration of \eqref{1541}, where $P_{\ul{X}} \in \overline{C_iB_{i+1}}$ for some $i$. Both red and both blue quadrangles have the same measure by Corollary \ref{883}  and the intercept theorem.}\label{fig:ub7a}
\end{figure}

\begin{figure}[H]
\centering
 \begin{tikzpicture}[scale=1.5]
   \foreach \s in {5} {
   \foreach \b in {0.35}{
   \foreach \q in {2} {
   \foreach \r in {0.3} {
   \foreach \g in {\b*(1-1/\q)} {

   \coordinate[label=below:{$C_0=(0,0)$}] (0) at (0,0);
   \coordinate[label=below:{$(1,0)$}] (10) at (\s,0);
   \coordinate[label=above:{$I=(0,1)$}] (01) at (0,\s);
   \coordinate[label=right:{$B_3=(1,1)$}] (11) at (\s,\s);
   \coordinate[label=left:{$G_{1}=\hat{P}_{\ul{X}}$}] (G1) at ({\s*(1*(1-\b)/(\q+1))},{\s*(1*(1-\b)/(\q+1)+\b)});
   \coordinate[label=right:{$G_{2}=\hat{P}^{'}$}] (G2) at ({\s*(2*(1-\b)/(\q+1))},{\s*(2*(1-\b)/(\q+1)+\b)});
   \draw[] (0) -- (10) -- (11) -- (01) -- cycle;

   \foreach \i in {1,...,\q} {
   \coordinate[label=above:{$A_{\i}$}] (A\i) at ({\s*\i*\r},{\s*(\i*\r+\q*(\b-\r))});
   \coordinate[label=below:{$B_{\i}$}] (B\i) at ({\s*(\r-\b+\i*\b)},{\s*(\r-\b+\i*\b)});
   \coordinate[label=right:{$C_{\i}$}] (C\i) at ({\s*\i*\b},\s*\i*\b);

   \draw (A\i) -- (B\i);
   \draw (A\i) -- (C\i);
   \draw (01) -- (A\i);
   };
   \coordinate[label=left:{$G_{0}$}] (G0) at (0,\s*\b);
   \coordinate[label=above:{$G_{3}$}] (G3) at ({\s*(1-\b)},\s);
   \coordinate[label=left:{$A_{0}$}] (A0) at (0,{\s*(\q*(\b-\r))});
   \coordinate[label=above:{$A_{3}$}] (A3) at ({\s*(\q+1)*\r},{\s*((\q+1)*\r+\q*(\b-\r))});

\coordinate[label=-45:$P_{\ul{X}}$] (P1) at ($(B1)!0.5!(C1)$);
   \coordinate[label=-45:$P^{'}$] (P2) at ($(P1)+({\s*\b},{\s*\b})$);
   \filldraw[fill=red] (B1) -- (P1) -- (A1) --  cycle;
   \filldraw[fill=red] (B2) -- (P2) -- (A2) --  cycle;
   \foreach \i in {A,B,C,G} {
   \foreach \j in {1,...,\q} {
   \fill (\i\j) circle(1pt);

   };
   };
   };
   };
   };
   };

   \fill (P1) circle(1pt);
   \fill (P2) circle(1pt);

   \fill (G0) circle(1pt);
   \fill (A0) circle(1pt);
   \fill (A3) circle(1pt);
   \fill (G3) circle(1pt);
   \fill (0) circle(1pt);
   \fill (10) circle(1pt);
   \fill (01) circle(1pt);
   \fill (11) circle(1pt);
   \draw (0) -- (11);
   \draw (A0) -- (A3);
   \draw (G0) -- (G3);
   \draw (P1) -- (A1);
   \draw (P2) -- (A2);
 };
\end{tikzpicture}
\caption{Illustration of \eqref{1541}, where $P_{\ul{X}} \in \overline{B_iC_i}$ for some $i$. Both red quadrangles have the same measure by Corollary \ref{883}.}\label{fig:ub7b}
\end{figure}

\section{Open problems}
We formulate the following conjecture in form of a problem because we are rather convinced that it is correct.
\begin{prob}
Prove that the ordering for Case b) presented in Section \ref{1122} and illustrated in Figure \ref{fig:ub5} defines an asymptotically optimal bandwidth numbering.
\end{prob}
A larger program is formulated in the second problem:
\begin{prob}
Find and study other interesting graph classes that allow a reduction to the unit square for the asymptotics and lead to interesting and non-trivial orderings on the unit square.
\end{prob}

\section*{Acknowledgement}
We are grateful for the financial support of the \emph{Deutsche Forschungsgemeinschaft} for the DFG graduate school 1505/2 {\it welisa}.

\bibliographystyle{elsarticle-num}

\end{document}